\documentclass[12pt]{article}

\usepackage{amssymb}
\usepackage{amsthm}
\usepackage{amsmath}
\usepackage{graphicx}
\usepackage{fullpage}
\usepackage{color}
 \numberwithin{equation}{section}

\theoremstyle{plain}
\newtheorem{thm}{Theorem}[section]
\newtheorem{cor}[thm]{Corollary}

\newtheorem{lem}[thm]{Lemma}
\newtheorem{prop}[thm]{Proposition}

\theoremstyle{definition}

\newtheorem{ex}[thm]{Example}

\theoremstyle{remark}

\newtheorem{rem}[thm]{Remark}

\newcommand{\N}{\mathbb{N}}
\newcommand{\R}{\mathbb{R}}

\newcommand{\E}{\mathbb{E}} 
\newcommand{\PP}{\mathbb{P}} 

\newcommand{\I}{\infty}
\newcommand{\bp}{\begin{proof}[\ensuremath{\mathbf{Proof}}]}
\newcommand{\ep}{\end{proof}}
\DeclareMathOperator{\tr}{tr}
\DeclareMathOperator{\diag}{diag}
\newcommand{\Sn}{S_n(\R)}
\newcommand{\Mn}{M_n(\R)}
\newcommand{\beps}{\beta_\epsilon}
\newcommand{\diam}{\text{diam}}
\newcommand{\dist}{\text{dist}}

\begin{document}


\title{ An eigenvalue problem for fully nonlinear elliptic equations with gradient constraints}

\author{Ryan Hynd\footnote{Department of Mathematics, University of Pennsylvania. Partially supported by NSF grant DMS-1301628.}}  

\maketitle

\begin{abstract}
We consider the problem of finding $\lambda\in \mathbb{R}$ and a function $u:\mathbb{R}^n\rightarrow\mathbb{R}$ that satisfy the PDE
$$
\max\left\{\lambda + F(D^2u) -f(x),H(Du)\right\}=0, \quad x\in \mathbb{R}^n.
$$
Here $F$ is elliptic, positively homogeneous and superadditive, $f$ is convex and superlinear, and $H$ 
is typically assumed to be convex.  Examples of this type of PDE arise in the theory of singular ergodic control.  
We show that there is a unique $\lambda^*$ for which the above equation has a solution $u$ with appropriate growth as 
$|x|\rightarrow \infty$.  Moreover, associated to $\lambda^*$ is a convex solution $u^*$ that has bounded second derivatives, provided
$F$ is uniformly elliptic and $H$ is uniformly convex.  It is unknown whether or not $u^*$ is unique up to an additive constant; however, 
we verify this is the case when $n=1$ or when $F, f,H$ are ``rotational." 
\end{abstract}


\section{Introduction}
The eigenvalue problem of singular ergodic control is to find a real number $\lambda$ and 
function $u:\R^n\rightarrow\R$ that satisfy the PDE
 \begin{equation}\label{LinEigProb}
\max\left\{\lambda -\Delta u -f(x),|Du|-1\right\}=0, \quad x\in \R^n.
\end{equation}
Here $Du=(u_{x_i})$ is the gradient of $u$, $\Delta u=\sum^{n}_{i=1}u_{x_ix_i}$ is the usual Laplacian, and $f$ is assumed to be convex and superlinear 
$$
\lim_{|x|\rightarrow \infty}\frac{f(x)}{|x|}=\infty.
$$
We call any such $\lambda$ an {\it eigenvalue}.  In previous work \cite{Hynd}, we showed there is a unique eigenvalue $\lambda^*\in \R$ such that the PDE \eqref{LinEigProb} admits a viscosity solution $u$ satisfying the growth condition
$$
\lim_{|x|\rightarrow \infty}\frac{u(x)}{|x|}=1. 
$$
Moreover, associated to $\lambda^*$, there is always one solution $u^*$ that is convex with 
$D^2u^*\in L^\I(\R^n;\Sn)$. Here $\Sn$ denotes the collection of real, symmetric $n\times n$ matrices. 

\par The eigenvalue $\lambda^*$ is also known to have the ergodic control theoretic interpretation
$$
 \lambda^*:=\inf_{\nu}\limsup_{t\rightarrow \infty}\frac{1}{t}\left\{\E\int^{t}_{0}f\left(\sqrt{2}W(s) + \nu(s)\right)ds + |\nu|(t)\right\}
$$
as shown in \cite{Menaldi}. Here $(W(t),t\ge 0)$ is an $n$-dimensional Brownian motion on a probability space $(\Omega, {\mathcal F}, \mathbb{P})$ 
and $\nu$ is an $\R^n$ valued control process. Each $\nu$ is required to be adapted to 
the filtration generated by $W$ and satisfy
$$
\begin{cases}
\nu(0)=0\;\\
t\mapsto \nu(t) \; \text{is left continuous}\\
|\nu|(t)<\infty , \;\text{for all}\; t> 0\;
\end{cases}
$$
$\PP$ almost surely; the notation $|\nu|(t)$ denotes the total variation of $\nu$ restricted to the interval $[0,t)$.  We say $\nu$ is a {\it singular control} as it may have sample paths that are not be absolutely continuous with respect to the standard Lebesgue measure on $[0,\infty)$.  We refer the reader to \cite{Borkar, Fleming, Oksendal} for more information on how PDE arise in singular stochastic control. 

\par We also showed in \cite{Hynd} that $\lambda^*$ is given by the following ``minmax" formula
\begin{equation}\label{minmaxLap}
\lambda^*=\inf\left\{\sup_{|D\psi(x)|<1}\left\{\Delta\psi(x) + f(x)\right\} :  \psi\in C^2(\R^n), \; \liminf_{|x|\rightarrow \infty}\frac{\psi(x)}{|x|}\ge 1 \right\}
\end{equation}
and the ``maxmin" formula
\begin{equation}\label{maxminLap}
\lambda^*=\sup\left\{\inf_{x\in\R^n}\left\{\Delta\phi(x) + f(x)\right\} :  \phi\in C^2(\R^n), \, |D\phi|\le 1 \right\}.
\end{equation}
The purpose of this paper is to verify generalizations of these results.

\par In particular, we consider the following eigenvalue problem: find $\lambda\in \R$ and $u:\R^n\rightarrow \R$ satisfying the PDE
\begin{equation}\label{EigProb}
\max\left\{\lambda + F(D^2u) -f(x),H(Du)\right\}=0, \quad x\in \R^n.
\end{equation}
Here $D^2u=(u_{x_ix_j})$ is the hessian of $u$.  A standing assumption in this paper is that the nonlinearity 
$F:\Sn\rightarrow \R$ is elliptic,  positively homogeneous, and superadditive:
\begin{equation}\label{Fassump}
\begin{cases}
-\Theta\tr N\le F(M+N)-F(M)\le -\theta\tr N, \quad (N\ge 0)
\\
F(tM)=tF(M)\
\\
F(M)+F(N)\le F(M+N)\
\end{cases}
\end{equation}
for each $M,N\in \Sn$, $t\ge 0$ and some $\theta,\Theta\ge 0$. If $\theta>0$, we say $F$ is uniformly elliptic.  For instance, in \eqref{LinEigProb} $F$ is the linear function $F(M)=-\tr M.$
And a more typical nonlinear example we have in mind is
$$
F(M)=\min_{1\le k\le N}\{-\tr(A_kM)\},
$$
where each $\{A_k\}_{k=1,\dots,N}\subset\Sn$ satisfies
$$
\theta|\xi|^2\le A_k\xi\cdot\xi\le \Theta|\xi|^2, \quad \xi\in \R^n.
$$

\par We will assume throughout that the gradient constraint function $H\in C(\R^n)$ satisfies 
\begin{equation}\label{Hassump}
\begin{cases}
H(0)<0\\
\{p\in \R^n: H(p)\le 0\}\; \text{is compact and strictly convex.}
\end{cases}
\end{equation}
In the motivating equation \eqref{LinEigProb}, $H(p)=|p|-1$. And in view of the results of \cite{Hynd}, it is natural 
to study solutions of \eqref{EigProb} subject to a suitable growth condition.  To this end, we define the function 
$$
\ell(v):=\max\{p\cdot v: H(p)\le 0\}, \quad v\in \R^n
$$
which is also known as the support function of the convex set $\{p\in \R^n: H(p)\le 0\}$.

\par  Note that we can replace $H$ in \eqref{EigProb} 
with the explicit convex gradient constraint 
$$
H_0(p):=\max_{|v|=1}\left\{p\cdot v -\ell(v)\right\}
$$
since $H(p)\le 0$ if and only if $H_0(p)\le 0$ (Theorem 8.24 in \cite{Rock}).  This is something we will do repeatedly in the work that follows. We also note 
that by the assumptions \eqref{Hassump}, there are positive constants $c_0, c_1$ such that 
\begin{equation}\label{UpperLowerell}
c_0|v|\le \ell(v)\le c_1|v|, \quad v\in \R^n
\end{equation}
and consequently
\begin{equation}\label{UpperLowerH}
|p|-c_1\le H_0(p)\le |p|-c_0, \quad p\in \R^n.
\end{equation}

\par The main result of this paper is as follows. 
\begin{thm}\label{Thm1}
Assume \eqref{Fassump}, \eqref{Hassump}, and that $f$ is convex and superlinear.  \\
(i) There is a unique $\lambda^*\in\R$ such that \eqref{EigProb} has a viscosity solution $u\in C(\R^n)$ 
satisfying the growth condition
\begin{equation}\label{ellgrowth}
\lim_{|x|\rightarrow \infty}\frac{u(x)}{\ell(x)}=1.
\end{equation}
Associated to $\lambda^*$ is a convex viscosity solution $u^*$ that satisfies \eqref{ellgrowth}. \\
(ii) Suppose that $F$ is uniformly elliptic, $H$ is convex and that there are $\sigma,\Sigma>0$ such that 
\begin{equation}\label{Hassump2}
\sigma|\xi|^2\le D^2H(p)\xi\cdot \xi\le \Sigma|\xi|^2, \quad \xi\in\R^n
\end{equation}
for Lebesgue almost every $p\in \R^n$. Then we may choose $u^*$ to satisfy $D^2u^*\in L^\I(\R^n;\Sn)$.
\end{thm}
\par When $\lambda=\lambda^*$ in \eqref{EigProb}, we will call solutions that satisfy the growth condition \eqref{ellgrowth} {\it eigenfunctions}.  It is unknown if eigenfunctions are unique up to an additive constant. However, we establish below that when $n=1$ any two convex eigenfunctions differ
by a constant; see Proposition \ref{UniqunessN1}.  We also show that if $F$, $f$ and $H$ are ``rotational," then $u^*$ can be chosen radial and twice continuously differentiable. This generalizes Theorem 2.3 of \cite{Kruk} and Theorem 1.3 of our previous work \cite{Hynd}. 

\begin{thm}\label{SymmRegThm}
Suppose 
\begin{equation}\label{SymmetryCond}
\begin{cases}
f(Ox)=f(x)\\
H(O^tp)=H(p)\\
F(OMO^t)=F(M)
\end{cases}
\end{equation}
for each $x,p\in \R^n$, $M\in \Sn$ and orthogonal $n\times n$ matrix $O$. If $F$ is uniformly elliptic and $H$ satisfies \eqref{Hassump2}, then there is a radial eigenfunction $u^*\in C^2(\R^n)$.
\end{thm}

\par In Proposition \ref{ConvUniqueness} below, we assume \eqref{SymmetryCond} and show any two convex, radial eigenfunctions differ by an additive constant.  
Unfortunately, we do not know if this symmetry assumption ensures that every eigenfunction is radial.  Finally, we verify a minmax formula for $\lambda^*$ which is the fully nonlinear analog of the formula \eqref{minmaxLap}. However, for nonlinear $F$, we only establish an inequality corresponding to the formula \eqref{maxminLap}. 
\begin{thm}\label{minmaxThm}
Define
$$
\lambda_+:=\inf\left\{\sup_{H(D\psi(x))<0}\left\{-F(D^2\psi(x)) + f(x)\right\} :  \psi\in C^2(\R^n), \; \liminf_{|x|\rightarrow \infty}\frac{\psi(x)}{\ell(x)}\ge 1 \right\}.
$$
and 
$$
\lambda_-:=\sup\left\{\inf_{x\in\R^n}\left\{-F(D^2\phi(x)) + f(x)\right\} :  \phi\in C^2(\R^n), \, H(D\phi)\le 0 \right\}
$$
Then 
$$
\lambda_-\le \lambda^*\le\lambda_+.
$$
If there is an eigenfunction $u^*$ that satisfies $D^2u^*\in L^\I(\R^n;\Sn)$, then 
$\lambda^*=\lambda_+.$
\end{thm}
The organization of this paper is as follows. In section \ref{CompSect}, we verify the uniqueness of eigenvalues as detailed in Theorem \ref{Thm1}. 
Then we consider the existence of an eigenvalue $\lambda^*$ in section \ref{ExistSec}. Next, we verify Theorem \ref{SymmRegThm} in section \ref{RegSect} and prove Theorem \ref{SymmRegThm} in section \ref{1DandRotSymmSect}. Section \ref{MinMaxSect} of this paper is dedicated to the proof of Theorem \ref{minmaxThm}.  Finally, we would like to acknowledge hospitality of the University of Pennsylvania's Center of Race $\&$ Equity in Education where part of this paper was written.

\section{Comparison principle}\label{CompSect}
In this section, we show there can be at most one eigenvalue as detailed in Theorem \ref{Thm1}.  As equation \eqref{EigProb} is a fully 
nonlinear elliptic equation for a scalar function $u$, we will employ the theory of viscosity solutions \cite{Bardi, Crandall, CIL, Fleming}. In particular, we will use results and notation from 
the ``user guide" \cite{CIL}.  Moreover, going forward we typically will omit the modifier ``viscosity" when we refer to sub- and supersolutions.  We begin our discussion with a basic proposition about subsolutions 
of the first order PDE $H(Du)=0$. 
\begin{lem}\label{HLipLem}
A function $u\in C(\R^n)$ satisfies 
\begin{equation}\label{HSubsoln}
H(Du(x))\le 0, \quad x\in \R^n
\end{equation}
if and only if 
\begin{equation}\label{ellLip}
u(x)-u(y)\le \ell(x-y), \quad x,y\in \R^n. 
\end{equation}
\end{lem}
\begin{proof}
Assume \eqref{HSubsoln}.  Then $u$ is Lipschitz by \eqref{UpperLowerH}, and $H(Du(x))\le 0$ for almost every $x\in \R^n$.  Let $u^\epsilon:=\eta^\epsilon*u$ be a standard mollification of $u$. That is, $\eta\in C^\infty_c(\R^n)$
is a nonnegative, radial function supported in $B_1(0)$ that satisfies $\int_{\R^n}\eta(z)dz=1$ and $\eta^\epsilon:=\epsilon^{-n}\eta(\cdot/\epsilon)$. It is readily verified that $u^\epsilon\in C^\infty(\R^n)$ and $u^\epsilon$ converges to $u$ uniformly as $\epsilon$ tends to 0; see Appendix C.5 of \cite{Evans2} 
for more on mollification.  As $H_0$ is convex, we have by Jensen's inequality 
$$
H_0(Du^\epsilon)=H_0\left(D(\eta^\epsilon*u)\right)=H_0\left(\eta^\epsilon*Du\right)\le \eta^\epsilon*H_0(Du)\le 0.
$$
It follows that for any $x,y\in \R^n$
$$
u^\epsilon(x)-u^\epsilon(y)=\int^1_0Du^\epsilon(y+t(x-y))\cdot (x-y)dt\le \ell(x-y). 
$$
Sending $\epsilon\rightarrow 0^+$ gives \eqref{ellLip}.

\par For the converse, suppose there is $p\in \R^n$ such that
$$
u(x)\le u(x_0)+p\cdot(x-x_0) + o(|x-x_0|)
$$
as $x\rightarrow x_0$.  Substituting  $x=x_0 - tv$ for $t>0$ and $|v|=1$ above gives
$$
u(x_0)- t\ell(v)\le u(x_0-tv)\le u(x_0) -t p\cdot v+o(t).
$$
As a result $p\cdot v\le \ell(v)$. As $v$ was arbitrary, $H(p)\le 0$.
\end{proof}
\begin{cor}
The function $\ell$ satisfies \eqref{HSubsoln}. Moreover, at any $x\in \R^n$ for which $\ell$ is differentiable 
$$
\ell(x)=D\ell(x)\cdot x\quad  \text{and} \quad H(D\ell(x))=0.
$$
\end{cor}
\begin{proof}
As $\ell$ is convex and positively homogeneous, it is sublinear. Therefore, $\ell(x)\le \ell(y)+\ell(x-y)$ for each $x,y\in \R^n$. By the previous lemma, $\ell$ satisfies \eqref{HSubsoln}.  
Now suppose
that $\ell$ is differentiable at $x$, and choose $\xi$ such that $H(\xi)\le 0$ and $\ell(x)=\xi\cdot x$. Then, as $y\rightarrow x$
\begin{align*}
\xi\cdot y& \le\ell(y) \\
& =\ell(x)+ D\ell(x)\cdot (y-x)+o(|y-x|)\\
& = \xi\cdot x+D\ell(x)\cdot (y-x)+o(|y-x|).
\end{align*}
Choosing $y=x+tv$, for $t>0$ and $v\in \R^n$ gives $\xi\cdot v\le D\ell(x)\cdot v + o(1)$ as $t\rightarrow 0^+$.  Thus, $\xi=D\ell(x)$ and $H(D\ell(x))\le 0$. If $x\neq 0$,
$$
H_0(D\ell(x))\ge D\ell(x)\cdot \frac{x}{|x|} -\ell\left(\frac{x}{|x|}\right)=\frac{D\ell(x)\cdot x-\ell(x)}{|x|}=0
$$
and so $H(D\ell(x))=0$. Conversely, if $x=0$, then $\ell$ is linear since it is positively homogeneous. However, this would contradict \eqref{UpperLowerell}. 
\end{proof}
The following assertion is a comparison principle for eigenvalues that makes use of the growth condition \eqref{ellgrowth}. 

\begin{prop}\label{lamCompProp}
Assume $u\in USC(\R^n)$ is a subsolution of \eqref{EigProb} with eigenvalue $\lambda$ and $v\in LSC(\R^n)$ is a 
supersolution of \eqref{EigProb} with eigenvalue $\mu$. If
\begin{equation}\label{growthComp}
\limsup_{|x|\rightarrow\infty}\frac{u(x)}{\ell(x)}\le 1\le \liminf_{|x|\rightarrow\infty}\frac{v(x)}{\ell(x)},
\end{equation}
the $\lambda\le \mu$. 
\end{prop}
\begin{rem}
Any subsolution $u$ of \eqref{EigProb} satisfies $H(Du)\le 0$. By Lemma \ref{HLipLem}, $u$ 
then satisfies \eqref{ellLip} and therefore the first inequality in \eqref{growthComp} automatically holds. 
We have included both inequalities in \eqref{growthComp} simply for aesthetic purposes, and we continue this practice throughout this paper. 
\end{rem}

\begin{proof} For $\tau\in (0,1)$ and $\eta>0$, set 
$$
w^{\tau}(x,y):=\tau u(x) - v(y), \quad \varphi^\eta(x,y):=\frac{1}{2\eta}|x-y|^2
$$
$x,y\in \R^n$.   Observe 
\begin{align}\label{wminusphi}
(w^{\tau}-\varphi^\eta)(x,y)&=\tau(u(x)-u(y))+\tau u(y)-v(y)-\frac{1}{2\eta}|x-y|^2 \nonumber \\
&\le \tau\ell(x-y)+\tau u(y)-v(y)-\frac{1}{2\eta}|x-y|^2 \nonumber \\
&\le \tau c_1|x-y|+\tau u(y)-v(y)-\frac{1}{2\eta}|x-y|^2 \nonumber \\
&\le \eta \tau^2 c_1^2+\tau u(y)-v(y)-\frac{1}{4\eta}|x-y|^2.
\end{align}
In view of \eqref{growthComp}, $\lim_{|y|\rightarrow \infty}(\tau u(y)-v(y))=-\infty$ and so
$$
\lim_{|x|+|y|\rightarrow \infty}(w^{\tau}-\varphi^\eta)(x,y)=-\infty.
$$
As a result, there is $(x_\eta,y_\eta)$ maximizing $w^{\tau}-\varphi^\eta$. 

\par By Theorem 3.2 in \cite{CIL}, for each $\rho>0$ there are $X,Y\in \Sn$ with $X\le Y$ such that 
$$
\left(\frac{x_\eta-y_\eta}{\eta},X\right)\in \overline{J}^{2,+}(\tau u)(x_\eta)
$$
and 
$$
\left(\frac{x_\eta-y_\eta}{\eta},Y\right)\in \overline{J}^{2,-}v(y_\eta).
$$
Note that 
\begin{align*}
H_0\left(\frac{x_\eta-y_\eta}{\eta}\right)&=H_0\left(\tau\frac{x_\eta-y_\eta}{\tau\eta}+(1-\tau)0\right)\\
&\le \tau H_0\left(\frac{x_\eta-y_\eta}{\tau\eta}\right)+(1-\tau)H_0(0)\\
&\le (1-\tau)H_0(0)\\
&<0.
\end{align*}
As $v$ is a supersolution of \eqref{EigProb}, 
$$
\mu + F(Y)-f(y_\eta)\ge 0.
$$
Since $F$ is elliptic and positively homogeneous, 
\begin{align}\label{ComparisonIneq}
\tau \lambda -\mu&\le -\tau F\left(\frac{X}{\tau}\right)+F(Y)+\tau f(x_\eta)-f(y_\eta) \nonumber\\
&= - F\left(X\right)+F(Y)+\tau f(x_\eta)-f(y_\eta)\nonumber \\
&\le \tau f(x_\eta)-f(y_\eta)\nonumber \\
&=  f(x_\eta)-f(y_\eta) + (\tau-1)f(x_\eta)\nonumber \\
&\le  f(x_\eta)-f(y_\eta) +(\tau-1)\inf_{\R^n}f.
\end{align}
\par We now claim that $(y_\eta)_{\eta>0}\subset\R^n$ is bounded. To see this, recall inequality \eqref{wminusphi}. If 
there is a sequence $\eta_k\rightarrow 0$ as $k\rightarrow \infty$ for which $|y_{\eta_k}|$ is unbounded, then $(w^\tau-\varphi^{\eta_k})(x_{\eta_k},y_{\eta_k})$ tends to $-\infty$ as $k\rightarrow\infty$. However, 
\begin{align*}
(w^\tau-\varphi^{\eta_k})(x_{\eta_k},y_{\eta_k})&=\sup_{\R^n\times\R^n}(w^\tau-\varphi^{\eta_k})\\
&\ge (w^\tau-\varphi^\eta)(0,0)\\
&=\tau u(0)-v(0).
\end{align*}
Thus, $(y_\eta)_{\eta>0}$ and similarly $(x_\eta)_{\eta>0}$ is bounded. It then follows from Lemma 3.1 in \cite{CIL} that 
$$
\lim_{\eta\rightarrow 0^+}\frac{|x_\eta-y_\eta|^2}{2\eta}=0
$$
and $(x_\eta,y_\eta)_{\eta>0}\subset\R^n\times\R^n$ has a cluster point $(x_\tau,x_\tau)$. Passing to the limit along an appropriate sequence $\eta$ tending to $0$ in \eqref{ComparisonIneq} then gives
\begin{equation}\label{ComparisonIneq2}
\tau \lambda-\mu\le (\tau-1)\inf_{\R^n}f.
\end{equation}
We conclude after sending $\tau\rightarrow 1^-$. 
\end{proof}
\begin{cor}
There can be at most one $\lambda\in \R$ for which \eqref{EigProb} has a solution $u$ satisfying \eqref{ellgrowth}. 
\end{cor}
\par We are uncertain whether or not eigenfunctions $u$ are uniquely defined up to an additive constant.  However, we do know that if $F$ is not uniformly elliptic and $f$ is not strictly convex, eigenfunctions are not necessarily unique. For instance when $F\equiv 0$ and $H(p)=|p|-1$, equation \eqref{EigProb} reduces to
\begin{equation}\label{FzeroEqn}
\max\{\lambda-f,|Du|-1\}=0, \quad \R^n.
\end{equation}
It is easily verified that $\lambda^*=\inf_{\R^n}f$ and $u(x)=|x-x_0|$
is a solution of \eqref{FzeroEqn} for each $x_0$ such that $\inf_{\R^n}f=f(x_0)$. Notice that if there is another point $y_0\neq x_0$ where 
$f$ attains its minimum, then $u(x)=|x-y_0|$ is another solution.  

\par We will give some conditions in Proposition \ref{UniqunessN1} below that guarantee uniqueness when $n=1$.  However, we postpone this discussion until after we have considered the regularity of solutions of \eqref{EigProb}. We conclude this section by giving a few examples with explicit solutions.  
\begin{ex}
Assume $n=1$, and consider the eigenvalue problem
$$
\begin{cases}
\max\{\lambda - u'' -x^2, |u'|-1\}=0, \quad x\in \R\\
\lim_{|x|\rightarrow \infty}\frac{u(x)}{|x|}=1
\end{cases}.
$$
Direct computation gives the explicit eigenvalue 
$$
\lambda^*=(2/3)^{2/3}
$$
with a corresponding eigenfunction
\begin{align*}
u^*(x)&=\inf_{|y|<(\lambda^*)^{1/2}}\left\{\frac{\lambda^*}{2}y^2-\frac{1}{12}y^4 +|x-y|\right\}\\
         &= 
        \begin{cases}
        \frac{\lambda^*}{2}x^2-\frac{1}{12}x^4, \quad |x|<(\lambda^*)^{1/2}\\
        \frac{\lambda^*}{2}[(\lambda^*)^{1/2}]^2-\frac{1}{12}[(\lambda^*)^{1/2}]^4 +(x-(\lambda^*)^{1/2}),\quad x\ge (\lambda^*)^{1/2}  \\ 
        \frac{\lambda^*}{2}[(\lambda^*)^{1/2}]^2-\frac{1}{12}[(\lambda^*)^{1/2}]^4 -(x+(\lambda^*)^{1/2}),\quad x\le -(\lambda^*)^{1/2}   \\ 
         \end{cases}.
\end{align*}
One checks additionally that $u^*\in C^2(\R)$.  In fact, searching for a solution that is twice continuously differentiable lead us to the particular value of $\lambda^*$. 
\end{ex}
\begin{ex}
The problem in the previous example can be generalized to any dimension $n\in \N$
\begin{equation}\label{SepVarProb}
\begin{cases}
\max\left\{\lambda - \Delta u -|x|^2, \max_{1\le i\le n}|u_{x_i}|-1\right\}=0, \quad x\in \R^n\\
\lim_{|x|\rightarrow \infty}u(x)/\sum^n_{i=1}|x_i|=1
\end{cases}.
\end{equation}
Note that this problem corresponds to \eqref{EigProb} when $F(M)=-\tr M$, $f(x)=|x|^2$ and $H(p)=\max_{1\le i\le n}|p_i|-1$. In this case, $\ell(v)=\sum^{n}_{i=1}|v_i|$.  Now assume $(\lambda_1, u_1)$ is a solution of the eigenvalue problem 
in the previous example. Then $\lambda^*=n\lambda_1$ and
$$
u^*(x)=\sum^{n}_{i=1}u_1(x_i)
$$
is a solution of the eigenvalue problem \eqref{SepVarProb} with $\lambda=\lambda^*$. Moreover, $u^*\in C^2(\R^n)$. 
\end{ex}

\section{Existence of an eigenvalue}\label{ExistSec}
In order to prove the existence of an eigenvalue, we will study solutions of the following PDE for $\delta>0$. 
\begin{equation}\label{deltaProb}
\max\left\{\delta u + F(D^2u) -f(x),H(Du)\right\}=0, \quad x\in \R^n.
\end{equation}
In particular, we will follow section 3 our previous work \cite{Hynd}, which was inspired by the approach of J. Menaldi, M. Robin and M. Taksar \cite{Menaldi}.  Employing the same techniques used to verify Proposition \ref{lamCompProp} above, we can establish the following assertion. 
\begin{prop}\label{DeltaCompProp}
Assume $\delta>0$, $u\in USC(\R^n)$ is a subsolution of \eqref{deltaProb} and $v\in LSC(\R^n)$ is a 
supersolution of \eqref{deltaProb}. If $u$ and $v$ satisfy \eqref{growthComp}, then $u\le v$. 
\end{prop}
It is now immediate that there can be at most one solution of \eqref{deltaProb} that satisfies 
the growth condition \eqref{ellgrowth}.  We will call this solution $u_\delta$. To verify that $u_\delta$ exists, we can appeal to Perron's method once we have appropriate sub and supersolutions. To this end, we first characterize the largest function $v$ that is less than a given function $g$ and satisfies $H(Dv)\le 0$. 
\begin{lem}\label{infConvLem}
Assume $g\in C(\R^n)$ is superlinear. The unique solution of the PDE 
\begin{equation}\label{DeterministicEq}
\max\{v-g,H(Dv)\}=0, \quad x\in \R^n
\end{equation}
that satisfies the growth condition \eqref{ellgrowth} is given by the inf-convolution of $g$ and $\ell$
\begin{equation}\label{vInfConv}
v(x):=\inf_{y\in\R^n}\left\{g(y)+\ell(x-y)\right\}
\end{equation}
\end{lem}
\begin{proof}
The uniqueness follows from Proposition \ref{DeltaCompProp}. In particular, this equation corresponds to \eqref{deltaProb} with $F\equiv 0$ and $\delta=1$.  Therefore, 
we only verify that $v$ given in \eqref{vInfConv} is a solution that satisfies the growth condition \eqref{ellgrowth}. Choosing $y=x$ gives, $v(x)\le g(x)$.  Also note $x\mapsto g(y) +\ell(x-y)$ satisfies \eqref{ellLip}, which implies that $v$ does as well. Hence, $v$ is a subsolution of \eqref{DeterministicEq}.
In particular, $\limsup_{|x|\rightarrow\infty}v(x)/\ell(x)\le 1$. Using $\ell(x-y)\ge \ell(x)-\ell(y)$,
$$
v(x)\ge \inf_{y\in\R^n}\left\{g(y)-\ell(y)\right\}+\ell(x).
$$
As $g$ is assumed superlinear, $\inf_{\R^n}\left\{g(y)-\ell(y)\right\}$ is finite. Thus, $\liminf_{|x|\rightarrow\infty}v(x)/\ell(x)\ge 1$.  

\par Finally, if $\psi$ is another subsolution of \eqref{DeterministicEq}
\begin{align*}
v(x)&=\inf_{y\in\R^n}\left\{g(y)+\ell(x-y)\right\}\\
&\ge\inf_{y\in\R^n}\left\{\psi(y)+\ell(x-y)\right\}\\
&\ge \psi(x).
\end{align*}
By Lemma 4.4 of \cite{CIL}, $v$ must be a supersolution of \eqref{DeterministicEq}.
\end{proof}
The solution of \eqref{DeterministicEq} when $g(x)=\frac{1}{2}|x|^2$ will be of particular interest to us and will help us construct a useful supersolution of PDE \eqref{deltaProb}. 

\begin{lem}\label{xsquaredLemma}
Let $g(x):=\frac{1}{2}|x|^2$ and $v$ the solution of \eqref{DeterministicEq} subject to the growth condition \eqref{ellgrowth}. Then
$$
v(x)=\frac{1}{2}|x|^2
$$
when $H(x)\le 0$, and  
$$
H(Dv)=0
$$
in $\{x\in \R^n: H(x)>0\}$
\end{lem}
\begin{proof}
Recall that $H(x)\le 0$ implies $\ell(v)\ge x\cdot v$ for all $v\in \R^n$. Thus
\begin{align*}
v(x)&=\inf_{y\in\R^n}\left\{\frac{1}{2}|y|^2+\ell(x-y)\right\}\\
&\ge\inf_{y\in\R^n}\left\{\frac{1}{2}|y|^2+x\cdot (x-y)\right\}\\
&=\inf_{y\in\R^n}\left\{\frac{1}{2}|y-x|^2+\frac{1}{2}|x|^2\right\}\\
&=\frac{1}{2}|x|^2.
\end{align*}
As $v(x)\le \frac{1}{2}|x|^2$ for all $x$, the first claim follows. 

\par Now suppose that $H(x)>0$. Then there is a $v_0\in \R^n$ with $|v_0|=1$ such that $\ell(v_0)< x \cdot v_0$. Fix $\epsilon>0$ so small
that $\ell(v_0)< x\cdot v_0-\epsilon$. Then 
\begin{align*}
v(x)&=\inf_{y\in\R^n}\left\{\frac{1}{2}|y-x|^2+\ell(y)\right\}\\
&\le \frac{1}{2}\left|(\epsilon v_0)-x\right|^2+\ell(\epsilon v_0)\\
&=\frac{1}{2}|x|^2 +\frac{\epsilon^2}{2}|v_0|^2-(\epsilon v_0)\cdot x+\ell(\epsilon v_0)\\
&=\frac{1}{2}|x|^2 +\frac{\epsilon^2}{2}+\epsilon[- v_0\cdot x+\ell(v_0)] \\
&\le \frac{1}{2}|x|^2 +\frac{\epsilon^2}{2}-\epsilon^2\\
&<\frac{1}{2}|x|^2.
\end{align*}
Since $v$ satisfies \eqref{DeterministicEq}, the PDE $H(Du)=0$ holds on the open set $\{x\in \R^n: H(x)>0\}$.
\end{proof}
We are now ready to exhibit sub and supersolutions of \eqref{deltaProb} that are 
comparable to $\ell(x)$ for large values of $|x|$. 
\begin{lem} Let $\delta\in (0,1)$. There are constants $K_1, K_2\ge 0$ such that 
\begin{equation}\label{uUpper}
\overline{u}(x)= \frac{K_1}{\delta}+\inf_{y\in\R^n}\left\{\frac{1}{2}|y|^2+\ell(x-y)\right\}
\end{equation}
is a supersolution of \eqref{deltaProb} satisfying \eqref{ellgrowth} and 
\begin{equation}\label{uLower}
\underline{u}(x)=(\ell(x)-K_2)^++\inf_{\R^n}f
\end{equation}
is a subsolution of \eqref{deltaProb} satisfying \eqref{ellgrowth}.  
\end{lem}

\begin{proof}
1. Choose 
$$
K_1:= - F(I_n) + \sup_{H(x)\le 0}f(x).
$$
Lemma \ref{xsquaredLemma} implies $\underline{u}(x)=\frac{K_1}{\delta}+\frac{1}{2}|x|^2$ when $H(x)\le 0$. Thus,
$$
\delta \underline{u}+F(D^2u)-f\ge K_1+F(I_n)-f\ge 0
$$
on $\{x\in\R^n: H(x)< 0\}$. 

\par We also have by Lemma \ref{xsquaredLemma} that $H(D\underline{u})=0$ on $\{x\in\R^n: H(x)> 0\}$. We will now verify that $H(D\underline{u}(x_0))=0$ when $H(x_0)=0$. 
To this end, suppose that 
$$
\underline{u}(x_0)+p\cdot (x-x_0)+o(|x-x_0|) \le \underline{u}(x)
$$
as $x\rightarrow x_0$. Using $\underline{u}(x_0)=\frac{K_1}{\delta}+\frac{1}{2}|x_0|^2$ and $\underline{u}(x)\le \frac{K_1}{\delta}+\frac{1}{2}|x|^2$ with the above inequality gives 
$$
\frac{1}{2}|x_0|^2+p\cdot (x-x_0)+o(|x-x_0|) \le \frac{1}{2}|x|^2,
$$
as $x\rightarrow x_0$. It follows that $p=x_0$, and so $H(p)=H(x_0)=0$.

\par 2. Choose $K_2\ge 0$ so large that 
$$
(\ell(x)-K_2)^+\le f(x)-\inf_{\R^n}f, \quad x\in \R^n. 
$$
Such a $K_2$ exists by the assumption that $f$ is superlinear and \eqref{UpperLowerell}. Observe $\underline{u}$ defined in \eqref{uLower} satisfies \eqref{ellLip}; thus 
$H(D\underline{u})\le 0$. And as $\ell$ is convex, $\underline{u}$ is convex. Therefore, $F(D^2\underline{u})\le 0$ and
$$
\delta\underline{u}+F(D^2\underline{u})-f\le \delta\underline{u}- f\le (\ell-K_2)^++\inf_{\R^n}f-f\le 0, \quad x\in \R^n
$$
for $\delta\le 1$.
\end{proof}
A key property of $u_\delta$ is that it is a convex function. This is critical to the arguments to follow. We also remark that our proof of this fact below was inspired 
by Korevaar's work \cite{KO} and is an adaption of Lemma 3.7 in \cite{Hynd}. The new feature we verify here is that the assumption that $F$ is superadditive still 
produces a convex solution. 

\begin{prop}\label{UdelConvex}
The function $u_\delta$ is convex. 
\end{prop}

\begin{proof}
For $\tau\in (0,1)$ and $\eta>0$, we define
$$
w^\tau(x,y,z):=\tau u(z) -\frac{u(x)+u(y)}{2}
$$
and 
$$
\varphi^\eta(x,y,z)=\frac{1}{2\eta}\left|\frac{x+y}{2}-z\right|^2
$$
for $x,y,z\in \R^n.$  Notice that
\begin{eqnarray}\label{SimpleEstW}
(w^\tau -\varphi_\eta)(x,y,z)&=& \tau \left\{u\left(z\right) - u\left(\frac{x+y}{2}\right)\right\} - \frac{1}{2\eta}\left|\frac{x+y}{2}-z\right|^2\nonumber \\
                                                             &  & + \tau u\left(\frac{x+y}{2}\right) - \frac{u(x) + u(y)}{2}\nonumber \\
                                                             &\le &\left( \tau\ell\left(\frac{x+y}{2}-z\right) - \frac{1}{2\eta}\left|\frac{x+y}{2}-z\right|^2\right) \nonumber  \\
                                                             && + \tau u\left(\frac{x+y}{2}\right)-\frac{u(x) + u(y)}{2}. 
\end{eqnarray}
By the growth condition \eqref{ellgrowth}, it follows that 
$$
\lim_{|x|+|y|\rightarrow \infty}\left\{\tau u\left(\frac{x+y}{2}\right)-\frac{u(x) + u(y)}{2}\right\}=-\infty
$$
and therefore
$$
\lim_{|x|+|y|+|z|\rightarrow \infty}(w^\tau-\varphi_\eta)(x,y,z)=-\infty.
$$
In particular, there is $(x_\eta,y_\eta,z_\eta)\in \R^n\times \R^n\times \R^n$ maximizing $w^\tau-\varphi_\eta.$ By Theorem 3.2 in \cite{CIL}, there are $X,Y,Z\in {\mathcal S}(n)$ such that 
\begin{equation}\label{3JetInc}
\begin{cases}
\left( -2D_x\varphi_\eta(x_\eta,y_\eta,z_\eta), X\right)\in \overline{J}^{2,-}u(x_\eta)\\
\left( -2D_{y}\varphi_\eta(x_\eta,y_\eta,z_\eta), Y\right)\in \overline{J}^{2,-}u(y_\eta)\\
\left( \frac{1}{\tau}D_{z}\varphi_\eta(x_\eta,y_\eta,z_\eta), Z\right)\in \overline{J}^{2,+}u(z_\eta)\\
\end{cases}
\end{equation}
and 
\begin{equation}\label{XYZineq}
\tau Z\le \frac{1}{2}(X+Y).
\end{equation}

\par Now set 
$$
p_\eta := -2D_x\varphi_\eta(x_\eta,y_\eta,z_\eta)=-2D_y\varphi_\eta(x_\eta,y_\eta,z_\eta)=D_z\varphi_\eta(x_\eta,y_\eta,z_\eta)=\frac{1}{\eta}\left(z_\eta -\frac{x_\eta+y_\eta}{2}\right).
$$
By the bottom inclusion in \eqref{3JetInc}, 
$$
\max\{\delta u(z_\eta) +F(Z) - f(z_\eta), H(p_\eta/\tau)\}\le 0.
$$
It follows that 
$$
H(p_\eta)=H\left(\tau\frac{p_\eta}{\tau}+(1-\tau)0\right)<0
$$
and by the top two inclusions in \eqref{3JetInc},
$$
\begin{cases}
\delta u(x_\eta) + F(X) - f(x_\eta)\ge 0\\
\delta u(y_\eta) + F(Y) - f(y_\eta)\ge 0
\end{cases}.
$$
Combining these inequalities with \eqref{XYZineq} gives

\begin{align}\label{usingfconvex}
\delta w^\tau(x,y,z)&\le \delta w(x_\eta,y_\eta,z_\eta) \nonumber \\
& = \tau\delta u(z_\eta) -\frac{\delta u(x_\eta)+\delta u(y_\eta)}{2} \nonumber \\
&\le \tau(-F(Z) + f(z_\eta)) -\frac{(-F(X)+f(x_\eta) ) +(-F(Y)+ f(y_\eta))}{2} \nonumber  \\
&= \left[-F(\tau Z) +\frac{F(X)+F(Y)}{2}\right] +\tau f(z_\eta)-\frac{f(x_\eta) + f(y_\eta)}{2}\nonumber \\
&\le \left[-F\left(\frac{X+Y}{2}\right) +\frac{F(X)+F(Y)}{2}\right] +\tau f(z_\eta)-\frac{f(x_\eta) + f(y_\eta)}{2}\nonumber \\
&\le f(z_\eta)-\frac{f(x_\eta) + f(y_\eta)}{2}+(\tau -1)f(z_\eta)\nonumber\\
&\le f(z_\eta)-\frac{f(x_\eta) + f(y_\eta)}{2}+(\tau -1)\inf_{\R^n}f
\end{align}
for each $(x,y,z)\in \R^n.$

\par   Another basic estimate for $w^\tau -\varphi_\eta$ that stems from \eqref{SimpleEstW} and \eqref{UpperLowerell} is
$$
(w^\tau-\varphi_\eta)(x,y,z)\le   \tau u\left(\frac{x+y}{2}\right) - \frac{u(x)+u(y)}{2} + \tau^2c_1^2\eta. 
$$
This inequality gives that $(x_\eta, y_\eta)_{\eta>0}\subset\R^n\times\R^n$ 
is bounded.  For were this not the case, $(w^\tau-\varphi_\eta)(x_\eta,y_\eta,z_\eta)$ tends to $-\infty$ yet 
\begin{eqnarray}
(w^\tau-\varphi_\eta)(x_\eta,y_\eta,z_\eta)&=&\max_{x,y,z}(w^\tau-\varphi_\eta)(x,y,z) \nonumber \\
&\ge &(w^\tau-\varphi_\eta)(0,0, 0) \nonumber \\
& =& (\tau -1) u(0) \nonumber \\
&>& -\infty, \nonumber
\end{eqnarray}
for each $\eta>0.$  Similarly,  $(z_\eta)_{\eta>0}\subset \R^n$ is bounded.  

\par Again we appeal to Lemma 3.1 in \cite{CIL}, which asserts the existence of a cluster point $(x_\tau,y_\tau, (x_\tau+ y_\tau)/2)$ of $((x_\eta, y_\eta,z_\eta))_{\eta>0}$ that maximizes 
$$
(x,y)\mapsto \tau u\left(\frac{x+y}{2}\right) - \frac{u(x)+u(y)}{2}. 
$$
Thus, we may pass to the limit through an appropriate sequence of $\eta$ tending to $0$ in \eqref{usingfconvex} to find for any $x,y\in\R^n$
$$
\tau u\left(\frac{x+y}{2}\right) - \frac{u(x)+u(y)}{2} \le f\left(\frac{x_\tau+y_\tau}{2}\right) - \frac{f(x_\tau)+f(y_\tau)}{2}+(\tau -1)\inf_{\R^n}f\le (\tau -1)\inf_{\R^n}f.
$$
Here we have used the convexity of $f$.  Finally, we conclude upon sending $\tau\rightarrow 1^-$. 
\end{proof}
By Aleksandrov's theorem (section 6.4 of \cite{Gariepy}), $u_\delta$ is twice differentiable at Lebesgue almost every $x\in\R^n$. At any such $x$, if $H(Du(x))<0$, then $x$ must be uniformly bounded for 
$$
f(x)=\delta u_\delta(x) +F(D^2u_\delta)\le \delta u_\delta(x).
$$
Recall that $f$ is superlinear and $u_\delta$ grows at most linearly. As precise statement is as follows.  
\begin{cor}\label{boundedDerOmega}
There is a constant $R$, independent of $\delta\in (0,1)$, such that if $p\in J^{1,-}u_\delta(x)$ and $H(p)<0$, then $|x|\le R$.
\end{cor}
\begin{proof}
As $u_\delta$ is convex, $J^{1,-}u_\delta(x)=\partial u(x)$; see proposition 4.7 in \cite{Bardi}.
It then follows that $(p,O_n)\in J^{2,-}u_\delta(x)$.  Thus, 
$$
\max\{\delta u_\delta(x)-f(x),H(p)\}\ge 0.
$$
As $H(p)<0$, it must be that $\delta u_\delta(x)-f(x)\ge 0$. As a result, 
$$
f(x)\le \delta u_\delta(x)\le K_1+\ell(x)\le K_1 +c_1|x|.
$$
Thus, $|x|\le R$ for some $R$ that is independent of $\delta\in (0,1)$. 
\end{proof}

Another important corollary is the following ``extension formula" for solutions.  We interpret this formula informally as: once the values of $u_\delta(x)$ are known for 
each $x$ satisfying $H(Du_\delta(x))<0$, $u_\delta$ is determined on all of $\R^n$. 
\begin{cor}\label{ExtCor}
Let
\begin{equation}\label{OmegaDel}
\Omega_\delta:=\R^n\setminus\{x\in \R^n: H(Du_\delta(x))\ge 0\;\text{in the viscosity sense}\;\}.
\end{equation}
Then
\begin{equation}\label{ExtensionForm}
u_\delta(x)=\inf\left\{u_\delta(y)+\ell(x-y): y\in \Omega_\delta\right\}, \quad x\in \R^n.
\end{equation}
Moreover, the infimum in \eqref{ExtensionForm} can be taken over $\partial\Omega_\delta$ when $x\notin\Omega_\delta$. 
\end{cor}
\begin{proof}
Set $u=u_\delta$ and define $v$ to be the right hand side of \eqref{ExtensionForm}. Since 
$u(x)\le u(y)+\ell(x-y)$ for each $x,y\in\R^n$, $u\le v$. If $x\in \overline{\Omega}_\delta$, 
there is a sequence $(x_k)_{k\in \N}\subset \Omega_\delta$ converging to $x$ as $k\rightarrow\infty$. Clearly, $v(x)\le u(x_k)+\ell(x-x_k)$ and sending $k\rightarrow\infty$ gives $v(x)\le u(x).$ Thus, $u(x)=v(x)$ for $x\in \overline{\Omega}_\delta$. 

\par Observe that $v(x)-v(y)\le \ell(x-y)$ for all $x,y\in \R^n$. Therefore, $v$ satisfies 
the PDE $H(Dv)\le 0$ on $\R^n.$ In particular, 
$$
\begin{cases}
H(Dv)\le 0\le H(Du), \quad & x\in \R^n\setminus\overline{\Omega}_\delta\\
v=u, \quad &  x\in\partial\Omega_\delta
\end{cases}
$$
while 
$$
\limsup_{|x|\rightarrow \infty}\frac{v(x)}{\ell(x)}\le 1\le \limsup_{|x|\rightarrow \infty}\frac{u(x)}{\ell(x)}.
$$
It follows from an argument similar to one given in Proposition \ref{lamCompProp} used to derive \eqref{ComparisonIneq2}, that
$$
\tau v-u\le (\tau- 1)\inf_{\R^n}f
$$
for each $\tau\in (0,1)$. In particular, $v\le u$ on $\R^n\setminus\overline{\Omega}_\delta$.
So we are able to conclude \eqref{ExtensionForm}.  

\par Now suppose $x\notin\Omega_\delta$ and choose $y\in\Omega_\delta$ such that 
$u(x)=u(y)+\ell(x-y)$.  There is a $t\in [0,1]$ such that 
$$
z=t y +(1-t)x \in \partial \Omega_\delta.
$$
Observe that since $u$ is convex and $\ell$ is positively homogeneous 
\begin{align*}
u(z)+\ell(x-z)&=u(t y +(1-t)x)+\ell(t(x-y))\\
&\le t(u(y) + \ell(x-y)) +(1-t)u(x)\\
&=tu(x)+(1-t)u(x)\\
&\le u(x).
\end{align*}
Thus, the minimum in \eqref{ExtensionForm} occurs on the boundary of $\partial \Omega_\delta$ when $x\notin \Omega_\delta$. 
\end{proof}
We will now verify the existence of an eigenvalue. Let $\delta\in (0,1)$ and $x_\delta$ denote a global minimizer of $u_\delta$
$$
\min_{x\in\R^n}u_\delta(x)=u_\delta(x_\delta).
$$
Clearly, $0\in J^{1,-}u(x_\delta)$ and by assumption $H(0)<0$; thus $x_\delta\in \Omega_\delta$. And by Corollary \ref{boundedDerOmega}, $|x_\delta|\le R$.  Set 
$$
\begin{cases}
\lambda_\delta:=\delta u_\delta(x_\delta)\\
v_\delta(x):=u_\delta(x)-u_\delta(x_\delta),\quad x\in \R^n
\end{cases}.
$$
In view of \eqref{uUpper}, \eqref{uLower},
\begin{equation}\label{LamDelBounds}
-\left(\inf_{\R^n}f\right)^{-}\le \lambda_\delta \le K_1 +\frac{1}{2}R^2;
\end{equation}
and by \eqref{UpperLowerell}
\begin{equation}\label{veeDelBounds}
\begin{cases}
0\le v_\delta(x)\le c_1(|x|+R)\\
|v_\delta(x)-v_\delta(y)|\le c_1|x-y|
\end{cases}
\end{equation}
for $x,y\in \R^n$ and $0<\delta<1$.
\begin{proof} (part $(i)$ of Theorem \ref{Thm1})
By \eqref{LamDelBounds} and \eqref{veeDelBounds}, there is a sequence of positive numbers $(\delta_k)_{k\in \N}$ tending to $0$, $\lambda^*\in \R$ and $u^*\in C(\R^n)$ such that $\lambda_{\delta_k}\rightarrow \lambda^*$ and $v_{\delta_k}\rightarrow u^*$ locally uniformly on $\R^n$. By the stability of viscosity solutions under locally uniform convergence (Lemma 6.1 in \cite{CIL}), $u^*$ satisfies \eqref{EigProb} with $\lambda=\lambda^*$. 

\par In view of the extension formula \eqref{ExtensionForm},
\begin{align*}
v_{\delta_k}(x)&=u_{\delta_k}(x)-u_{\delta_k}(x_{\delta_k})\\
&=\inf_{y\in \Omega_{\delta_k}}\{u_{\delta_k}(y)-u_{\delta_k}(x_{\delta_k})+\ell(x-y)\}\\
&\ge \inf_{y\in \Omega_{\delta_k}}\{\ell(x-y)\}\\
&\ge \inf_{y\in \Omega_{\delta_k}}\{\ell(x)-\ell(y)\}\\
&=\ell(x)-\sup_{y\in\Omega_{\delta_k}}\ell(y)\\
&\ge \ell(x)-\sup_{|y|\le R}\ell(y).
\end{align*}
Thus, $u^*(x)\ge  \ell(x)-\sup_{|y|\le R}\ell(y)$ and in particular, $u^*$ satisfies the growth 
condition \eqref{ellgrowth}. It now follows that $\lambda^*$ is the desired eigenvalue. 
\end{proof}
We now have the following characterization of the eigenvalue $\lambda^*$.  See also 
\cite{Armstrong} for a similar characterization of eigenvalues of operators that are uniformly elliptic, fully nonlinear, and positively homogeneous.  
\begin{cor}
Let $\lambda^*$ be as described in part $(i)$ of Theorem \ref{Thm1}. Then 
\begin{align}\label{LamChar1}
\lambda^*&=\sup\{\lambda\in \R: \text{there is a subsolution $u$ of \eqref{EigProb} with eigenvalue $\lambda$} \nonumber \\
&\left.\hspace{1in} \text{satisfying}\; \limsup_{|x|\rightarrow \infty}\frac{u(x)}{\ell(x)}\le 1\right\}.
\end{align}
and
\begin{align}\label{LamChar2}
\lambda^*&=\inf\{\mu\in \R: \text{there is a supersolution $v$ of \eqref{EigProb} with eigenvalue $\mu$} \nonumber \\
&\left.\hspace{1in} \text{satisfying}\; \liminf_{|x|\rightarrow \infty}\frac{v(x)}{\ell(x)}\ge 1\right\}.
\end{align}
\end{cor}
In particular, choosing $\lambda=\inf_{\R^n}f$ and $u\equiv 0$ in \eqref{LamChar1} gives $\lambda^*\ge \inf_{\R^n}f$. And selecting 
$\mu=-F(I_n)+\sup_{H(x)\le 0}f(x)$ and $v(x)=\inf_{\R^n}\{|y|^2/2+\ell(x-y)\}$ in \eqref{LamChar2} gives $\lambda^*\le -F(I_n)+\sup_{H(x)\le 0}f(x)$. In 
summary, we have the bounds on $\lambda^*$
$$
\inf_{x\in\R^n}f(x)\le \lambda^*\le-F(I_n)+\sup_{H(x)\le 0}f(x).
$$

\section{Regularity of solutions}\label{RegSect}
Our goal in this section is to prove part $(ii)$ of Theorem \ref{Thm1}. To this end, we will assume that $F$ is uniformly elliptic, assume $H$ satisfies \eqref{Hassump2} 
and derive a uniform upper bound on $D^2u_\delta$.   Recall $u_\delta$ is the unique solution of \eqref{deltaProb} that satisfies \eqref{ellgrowth}.  We will first use an easy semiconcavity argument to bound $D^2u_\delta(x)$ for all large values of $|x|$.  Then we will 
pursue second derivatives bounds on $u_\delta$ for smaller values of $|x|$. To this end, we will employ to the so-called ``penalty method" introduced by L. C. Evans \cite{Evans}.  For other related work, consult also \cite{HyndMawi,Ishii, Soner, Wiegner}.

\subsection{Preliminaries}
 An important identity for us will be 
\begin{equation}\label{ellFormula}
\ell(v)=\inf_{\lambda>0}\lambda H^*\left(\frac{v}{\lambda}\right), \quad v\in \R^n\setminus\{0\}
\end{equation}
where $H^*(w)=\sup_{p\in \R^n}\{p\cdot w - H(p)\}$ is the Legendre transform of $H$; see exercise 11.6 of \cite{Rock}.  This formula 
is crucial to our method for deriving second derivates estimates on $u_\delta$ for large values of $|x|$. 

\begin{lem}\label{W2infFarOutBound}
Define $\Omega_\delta$ as in \eqref{OmegaDel}. There is a constant $C$ such that 
$$
D^2u_\delta(x)\le \frac{C}{\text{dist}(x,\Omega_\delta)}I_n
$$
for Lebesgue almost every $x\in \R^n\setminus\overline{\Omega}_\delta$. 
\end{lem}
\begin{proof}
We will employ formula \eqref{ellFormula}. We will also use that
\begin{equation}\label{Hstar}
H^*(0)>0
\end{equation}
and 
\begin{equation}\label{Hstar2}
\frac{1}{\Sigma}|\xi|^2\le D^2H^*(w)\xi\cdot \xi\le \frac{1}{\sigma}|\xi|^2,\quad \xi\in\R^n
\end{equation}
for almost every $w\in \R^n$.   Let $v\in \R^n\setminus\{0\}$ and $\lambda>0$. Note \eqref{Hstar2} implies 
\begin{equation}\label{lowerHstar}
 \lambda H^*(0)+DH^*(0)\cdot v +\frac{1}{2\Sigma \lambda}|v|^2\le \lambda H^*\left(\frac{v}{\lambda}\right)\le \lambda H^*(0)+DH^*(0)\cdot v +\frac{1}{2\sigma \lambda}|v|^2.
\end{equation}
Thus, $\lim_{\lambda\rightarrow 0^+}\lambda H^*\left(v/\lambda\right)=+\infty$. And with \eqref{Hstar}, we also conclude that $\lim_{\lambda\rightarrow \infty}\lambda H^*\left(v/\lambda\right)=+\infty.$  As $\lambda\mapsto \lambda H^*\left(v/\lambda\right)$ is strictly convex, there is a unique $\lambda=\lambda(v)>0$ for which  $\ell(v)=\lambda(v) H^*(v/\lambda(v))$.  

\par Using the positive homogeneity of $\ell$, for $t>0$
\begin{align*}
\lambda(tv) H^*\left(\frac{tv}{\lambda(tv)}\right)&=\ell(tv)\\
&=t\ell(v)\\
&=t\lambda(v) H^*\left(\frac{v}{\lambda(v)}\right)\\
&=t\lambda(v) H^*\left(\frac{tv}{t\lambda(v)}\right).
\end{align*}
Thus, $\lambda(tv)=t\lambda(v)$. It also follows from \eqref{lowerHstar} that 
$$
\gamma:=\inf_{|v|=1}\lambda(v)>0.
$$
In particular, $\lambda(v)\ge \gamma |v|$, for each $v\neq 0.$

\par Again let $v\neq 0$, and choose $h\in \R^n$ so small that $v\pm h\neq 0$. Then for $\lambda=\lambda(v)$
\begin{align*}
\ell(v+h)-2\ell(v)+\ell(v-h)&\le \lambda H^*\left(\frac{v+h}{\lambda}\right) - 2\lambda H^*\left(\frac{v}{\lambda}\right)+\lambda H^*\left(\frac{v-h}{\lambda}\right)\\
&=\lambda\left[ H^*\left(\frac{v}{\lambda}+\frac{h}{\lambda}\right)-2H^*\left(\frac{v}{\lambda}\right)+H^*\left(\frac{v}{\lambda}-\frac{h}{\lambda}\right) \right]\\
&\le \lambda \frac{1}{\sigma}\left|\frac{h}{\lambda}\right|^2\\
&=\frac{1}{\sigma \lambda}|h|^2\\
&\le \frac{1}{\gamma \sigma|v|}|h|^2.
\end{align*}

\par Now we can employ the extension formula \eqref{ExtensionForm}. Let $x\in \R^n\setminus\overline{\Omega}_\delta$ and choose $h$ so small that $x\pm h\in\R^n\setminus\overline{\Omega}_\delta$. 
Selecting $y\in \partial\Omega_\delta$ so that $u_\delta(x)=u_\delta(y)+\ell(x-y)$ gives
\begin{align*}
u_\delta(x+h)-2u_\delta(x)+u_\delta(x-h)&\le \ell(x-y+h)-2\ell(x-y)+\ell(x-y-h)\\
&\le \frac{1}{\gamma \sigma |x-y|}|h|^2\\
&\le \frac{C }{\text{dist}(x,\partial \Omega_\delta)}|h|^2.
\end{align*}
The claim follows as $u_\delta$ is differentiable Lebesgue almost everywhere. 
\end{proof}
In order to complete the proof of part $(ii)$ of Theorem \ref{Thm1}, we must bound the second derivatives on $u_\delta$ on some subset of $\R^n$ that includes $\overline{\Omega}_\delta$. Before we detail our approach, it will be necessary for us to differentiate (a smoothing) of $F$. To this end, we extend $F$ 
to the space $\Mn$ of all $n\times n$ real matrices as follows 
$$
\overline{F}(M):=F\left(\frac{1}{2}(M+M^t)\right), \quad M\in \Mn. 
$$
We can then treat $\overline{F}(M)$ as a function of the $n^2$ real entries of the matrix $M\in \Mn$.  It is readily checked that $\overline{F}$ is uniformly elliptic, positively homogeneous and superadditive on  $\Mn$. In particular, $\overline{F}$ satisfies \eqref{Fassump} for each $M,N\in \Mn$ and $t\ge 0$.  This allows us to identify $F$ with $\overline{F}$ and we shall do this for the remainder of this section. 
\par We now define $F^\varrho$ as the standard mollification of $F$ 
$$
F^\varrho(M):=\int_{\Mn}\eta^\varrho(N)F(M-N)dN, \quad M\in \Mn. 
$$
The integral above is over the $n^2$ real variables $N=(N_{ij})\in \Mn$, and as in Lemma \ref{HLipLem}, $\eta\in C^\infty_c(\Mn)$ is a nonnegative function that is supported in $\{M\in \Mn: |M|\le 1\}$ and 
$\eta(M)$ only depends on $|M|$. Moreover, $\eta$ satisfies $\int_{\R^n}\eta(Z)dZ=1$ and we have defined $\eta^\varrho:=\varrho^{-n^2}\eta(\cdot/\varrho)$. See also section 4 of \cite{HyndMawi} or Proposition 9.8 in \cite{CC} for more details on mollifying functions of matrices.

\par It is readily verified that $F^\varrho\in C^\I(\Mn)$ and, with the help of \eqref{Fassump}, $F^\varrho$ is uniformly elliptic, concave and satisfies 
\begin{equation}\label{FFvarrhoEst}
F^\varrho(M)\le F(M)\le F^\varrho(M)+\sqrt{n}\Theta \varrho, \quad M\in \Mn. 
\end{equation}
However, $F^\varrho$ is not in general positively homogeneous. Nevertheless, $F^\varrho$ inherits a certain almost homogeneity property.
\begin{lem}\label{HomogeneityLEM}
For every $M\in \Mn$, 
$$
F^\varrho(M)=F^\varrho_{M_{ij}}(M)M_{ij} - \int_{\Mn}\eta^\varrho(N)F_{M_{ij}}(M-N)N_{ij}dN.
$$
In particular,
\begin{equation}\label{FrhoAlmostHomo}
|F^\varrho(M)-F^\varrho_{M_{ij}}(M)M_{ij}|\le \sqrt{n}\Theta\varrho, \quad M\in \Mn.
\end{equation}
\end{lem}
\begin{proof}
By the ellipticity assumption \eqref{Fassump}, $F$ is Lipschitz continuous. Rademacher's Theorem then implies
that $F$ is differentiable for Lebesgue almost every $M\in \Mn$, which we identify with $\R^{n^2}$. Therefore, 
$$
F^\varrho_{M_{ij}}(M)=\int_{\Mn}\eta^\varrho(N)F_{M_{ij}}(M-N)dN.
$$
See Theorem 1 of section 5.3 in \cite{Evans2} for an easy verification of this equality.  Since 
$F$ is positively homogenous of degree one, 
$$
F(M)=F_{M_{ij}}(M)M_{ij}
$$
for Lebesgue almost every $M\in \Mn$. And therefore, 
\begin{align*}
F^\varrho_{M_{ij}}(M)M_{ij}&=\int_{\Mn}\eta^\varrho(N)F_{M_{ij}}(M-N)M_{ij}dN\\
&=\int_{\Mn}\eta^\varrho(N)F_{M_{ij}}(M-N)(M_{ij}-N_{ij})dN\\
&\quad +\int_{\Mn}\eta^\varrho(N)F_{M_{ij}}(M-N)N_{ij}dN\\
&=\int_{\Mn}\eta^\varrho(N)F(M-N)dN+\int_{\Mn}\eta^\varrho(N)F_{M_{ij}}(M-N)N_{ij}dN\\
&=F^\varrho(M)+\int_{\Mn}\eta^\varrho(N)F_{M_{ij}}(M-N)N_{ij}dN.
\end{align*}
\par The ellipticity assumption \eqref{Fassump} also implies
$$
-\Theta|\xi|^2\le F_{M_{ij}}(M)\xi_i\xi_j\le-\theta|\xi|^2, \quad \xi\in \R^n
$$
for almost every $M\in \Mn$. Therefore, 
\begin{align*}
\left|\int_{\Mn}\eta^\varrho(N)F_{M_{ij}}(M-N)N_{ij}dN\right|&\le\int_{\Mn}\eta^\varrho(N)\left|F_{M_{ij}}(M-N)N_{ij}\right|dN \\
&\le \int_{\Mn}\eta^\varrho(N)\sqrt{\sum^n_{ij=1}\left(F_{M_{ij}}(M-N)\right)^2}\; |N|dN\\
&\le \sqrt{n}\Theta \int_{\Mn}\eta^\varrho(N)|N|dN\\
& = \sqrt{n}\Theta\varrho \int_{|Z|\le 1}\eta(Z)|Z|dZ \quad (Z=N/\varrho)\\
& \le \sqrt{n}\Theta\varrho \int_{|Z|\le 1}\eta(Z)dZ \\
& = \sqrt{n}\Theta\varrho.
\end{align*}
\end{proof}

\par We will additionally need to smooth out $H$ and $f$, and we will do so by using the standard mollifications 
 $H^\varrho=\eta^\varrho*H$ and $f^\varrho=\eta^\varrho*f$. Here $\eta$ is a standard mollifier on $\R^n$. We also select $\varrho_1$ so small that 
 \begin{equation}\label{Hsmallvarrho1}
 H^\varrho(0)<0, \quad \varrho\in (0,\varrho_1).
 \end{equation}
 The following lemma asserts that the solution of the PDE \eqref{deltaProb} is well approximated by a solution of the same equation with $H^\varrho$ and $f^\varrho$ replacing $H$ and $f$.

 \begin{lem}\label{firstApproxLem}
Assume $\delta\in (0,1)$ and $\varrho\in (0,\varrho_1)$. Let $u_{\delta,\varrho}$ be solution of \eqref{deltaProb} with $F, H^\varrho$, and $f^\varrho$ subject to the growth condition \eqref{ellgrowth} with 
$\ell^\varrho(v)=\sup\{p\cdot v: H^\varrho(p)\le 0\}$ replacing $\ell$. Then $\lim_{\varrho \rightarrow 0^+} u_{\delta,\varrho}=u_\delta$
locally uniformly on $\R^n$. 
\end{lem}
\begin{proof} Using test functions as in \eqref{uUpper} and \eqref{uLower} that correspond to \eqref{deltaProb} with $F, H^\varrho$, and $f^\varrho$ we find
$$
\inf_{\R^n}f^\varrho\le u_{\delta,\varrho}\le \frac{1}{\delta}\left(-F(I_n) + \sup_{H^\varrho\le 0}f^\varrho\right)+\ell^\varrho.
$$
By the convexity of $H$ and $f$, Jensen's inequality implies $H\le H^\varrho$ and $f\le f^\varrho$.  It then follows that $\ell^\varrho\le \ell$. By the ellipticity of $F$, $-F(I_n)\le n\Theta$ and so
\begin{equation}\label{udeltarhoBounds}
\inf_{\R^n}f\le u_{\delta,\varrho}\le \frac{1}{\delta}\left(n\Theta + \sup_{H\le 0}f^\varrho\right)+\ell.
\end{equation}
Since $f^\varrho\rightarrow f$ locally uniformly on $\R^n$, $u_{\delta,\varrho}$ is locally bounded on $\R^n$ independently of $\varrho\in (0,\varrho_1)$.  

\par Also notice that $H(Du_{\delta,\varrho})\le H^\varrho(Du_{\delta,\varrho})\le 0$ which implies that $u_{\delta,\varrho}$ is uniformly equicontinuous on $\R^n$. 
It follows that for each sequence of positive numbers $(\varrho_k)_{k\in \N}$ tending to 0, there is a subsequence of $(u_{\delta,\varrho_k})_{k\in \N}$ converging locally 
uniformly to some $u\in C(\R^n)$.  By the stability of viscosity solutions under local uniform convergence, $u$ is a solution of \eqref{deltaProb}.  In order to conclude, it suffices to verify that 
$u$ satisfies \eqref{ellgrowth}. Then by uniqueness we would have $u=u_\delta$ and the full sequence $(u_{\delta,\varrho_k})_{k\in \N}$ must converge to $u_\delta$. 

\par We now employ  the extension formula \eqref{ExtensionForm} with 
$$
\Omega_{\delta,\varrho}:=\R^n\setminus\{x\in \R^n: H(Du_{\delta,\varrho}(x))\ge 0\;\text{in the viscosity sense}\;\}
$$
to get 
\begin{align}\label{Lowerudeltarhobound}
u_{\delta,\varrho}(x)&=\inf_{y\in \Omega_{\delta,\varrho}}\{u_{\delta,\varrho}(x)+\ell^\varrho(x-y) \}\nonumber \\
&\ge\inf_{y\in \Omega_{\delta,\varrho}}\left\{ \inf_{\R^n}f+\ell^\varrho(x)-\ell^\varrho(y)\right\} \nonumber \\
&= \inf_{\R^n}f+\ell^\varrho(x)-\sup_{y\in \Omega_{\delta,\varrho}}\ell(y).
\end{align}
It is immediate from the proof of Corollary \eqref{boundedDerOmega} that there is an $R>0$ such that $\Omega_{\delta,\varrho}\subset B_R(0)$ for each $\delta\in (0,1)$ and $\varrho\in (0,\varrho_1)$.  We also leave it to 
the reader to verify that $\ell(v)=\lim_{\varrho\rightarrow 0^+}\ell^\varrho(v)$ for each $v\in \R^n$.  Passing to the limit along an appropriate sequence of $\varrho$ tending to 0 in \eqref{Lowerudeltarhobound} gives
$$
u(x)\ge  \inf_{\R^n}f+\ell(x)-\sup_{|y|\le R}\ell(y).
$$
Hence, $u$ satisfies \eqref{ellgrowth}. 
\end{proof}

\subsection{The penalty method}
 Now we fix $\delta\in (0,1)$, $\rho\in (0,\rho_1)$ and a choose a ball $B=B_R(0)\subset\R^n$ so large that 
\begin{equation}\label{BbigEnough}
H(p)\le 0\quad  \Longrightarrow\quad |p|\le R. 
\end{equation}
For $\epsilon>0$, we will now focus on solutions of the fully nonlinear PDE 
\begin{equation}\label{PenalizedEqn}
\delta u+F^\varrho(D^2u)+\beta_\epsilon(H^\varrho(Du))=f^\varrho, \quad x\in B
\end{equation}
subject to the boundary condition 
\begin{equation}\label{PenalizedBC}
u(x)=u_{\delta, \varrho}(x), \quad x\in \partial B.
\end{equation}
Recall that $u_{\delta,\varrho}$ is the solution of \eqref{deltaProb} with $F, H^\varrho$, and $f^\varrho$ subject to the growth condition \eqref{ellgrowth} with 
$\ell^\varrho(v)=\sup\{p\cdot v: H^\varrho(p)\le 0\}$ instead of $\ell$.

\par In \eqref{PenalizedEqn}, $F^\varrho$ is a standard mollification of $F$ and the family $\{\beta_\epsilon\}_{\epsilon> 0}$ of functions each satisfy 
\begin{equation}\label{betaAss}
\begin{cases}
\beta_\epsilon\in C^\infty(\R)\\
\beta_\epsilon(z)=0, \quad z\le 0\\
\beps(z)>0, \quad z>0\\
\beps'\ge0,\\
\beps''\ge0,\\
\beps(z)=(z-\epsilon)/\epsilon, \quad z\ge 2\epsilon\\
\end{cases}.
\end{equation}
Our intuition is that $\beta_\epsilon$ is a smoothing of Lipschitz function $z\mapsto (z/\epsilon)^+$; and therefore, 
solutions of \eqref{PenalizedEqn} should be close to solutions of $\max\{\delta u+F^\varrho(D^2u)-f^\varrho,H^\varrho(Du)\}=0$ that satisfy \eqref{PenalizedBC}. These solutions will in turn be very close to $u_{\delta,\varrho}|_{B}$ for $\varrho$ small (see Lemma \ref{FFvarrholocLem} below).

\par By a theorem of N. Trudinger (Theorem 8.2 in \cite{Trudinger}) there is a unique classical solution $u^\epsilon\in C^\I(B)\cap C(\overline{B})$ solving \eqref{PenalizedEqn} and satisfying the boundary condition \eqref{PenalizedBC}. This result relies on the Evans--Krylov a priori estimates for solutions of concave, fully nonlinear elliptic equations and the continuity method \cite{EvansC2, KrylovC2}. 
 Along with the concavity of $F$, the main structural condition that allows us to apply this theorem is that $p\mapsto \beta_\epsilon(H^\varrho(p))$ grows at most quadratically for each $\epsilon>0$.  We remark that $u^\epsilon$ naturally depends on the other parameters $\delta\in (0,1)$ and $\varrho\in (0,\varrho_1)$; we have chosen not to indicate this dependence for ease of notation.

\par Since $u_{\delta,\varrho}$ solves \eqref{deltaProb} with $F, H^\varrho$, and $f^\varrho$, we have from \eqref{betaAss} and \eqref{FFvarrhoEst} that 
\begin{align*}
\delta u_{\delta,\varrho}+F^\varrho(D^2u_{\delta,\varrho}) + \beta_\epsilon(H^\varrho(Du_{\delta,\varrho}))&= \delta u_{\delta,\varrho}+F^\varrho(D^2u_{\delta,\varrho})\\
&\le\delta u_{\delta,\varrho}+ F(D^2u_{\delta,\varrho})\\
&\le f^\varrho. 
\end{align*}
In view of \eqref{PenalizedEqn} and \eqref{PenalizedBC}, $u_{\delta,\varrho}\le u^\epsilon$ by a routine maximum principle argument. Also note
$$
F^\varrho(D^2u^\epsilon)\le f^\varrho -\delta u^\epsilon\le f^\varrho -\delta u_{\delta,\varrho}.
$$
The Aleksandrov-Bakelman-Pucci estimate (Theorem 3.6 in \cite{CC}, Theorem 17.3 in \cite{Gilbarg}) then implies 
$$
\sup_{B}u^\epsilon\le C\left(\sup_{\partial B}|u_{\delta,\varrho}|+\sup_{B}|f^\varrho-\delta  u_{\delta,\varrho}|\right)
$$
for some constant $C=C(\text{diam}(B),n,\theta,\Theta)$.  Combined with \eqref{udeltarhoBounds} and \eqref{BbigEnough}, we have the following supremum norm bound
$$
|u^\epsilon|_{L^\infty(B)}\le C\left\{\left(\inf_{\R^n}f\right)^- + \sup_B\ell + \frac{1}{\delta}\left(n\Theta+\sup_B|f^\varrho|\right)\right\}.
$$
We will use this estimate to obtain bounds on the higher derivatives of $u^\epsilon$ that will be independent of all $\epsilon>0$ and sufficiently small.  

We are now in a position to derive uniform estimates on the derivatives of $u^\epsilon$. We will borrow from the recent work by the author and H. Mawi on fully nonlinear elliptic equations with convex gradient 
constraints \cite{HyndMawi}. Note however, one of the main assumptions in \cite{HyndMawi} is that the nonlinearity is uniformly elliptic and 
{\it convex}; note the class of nonlinearities we study in this paper satisfy \eqref{Fassump} and are {\it concave}. We will make use of Lemma \ref{HomogeneityLEM} instead of a convexity assumption on $F$.  

\par We will also employ the uniform convexity assumption \eqref{Hassump2}, which implies
\begin{align}\label{Coercive}
\begin{cases}
H^\varrho(p)\ge H^\varrho(0)+DH^\varrho(0)\cdot p+\frac{\sigma}{2}|p|^2\\
DH^\varrho(p)\cdot p-H^\varrho(p)\ge -H^\varrho(0)+\frac{\sigma}{2}|p|^2\\
|DH^\varrho(p)|\le |DH^\varrho(0)| + \sqrt{n}\Sigma|p|
\end{cases}(p\in \R^n).
\end{align}
And we choose $\varrho_1>0$ sufficiently smaller if necessary so that \eqref{Hsmallvarrho1} holds and
$$
\begin{cases}
|H^\varrho(0)|\le |H(0)|+1\\
|DH^\varrho(0)|\le |DH(0)|+1\\
|f^\varrho|_{W^{1,\infty}(B)}\le |f|_{W^{1,\infty}(B_{R+1}(0))} \quad (B=B_R(0))
\end{cases}
$$
for $0<\varrho<\varrho_1$. In stating our uniform estimates below, it will be convenient for us to label the following list  
$$
\Pi:=\left(\sigma,\Sigma,\theta,\Theta, n,\diam(B), H(0),|DH(0)|, |f|_{W^{1,\infty}(B_{R+1}(0))}|, \inf_{\R^n}f,\sup_B\ell, \varrho_1\right).
$$

\begin{lem}
Let $\delta\in(0,1)$, $\varrho\in (0,\varrho_1)$, $\epsilon\in (0,1)$ and suppose $\zeta\in C^\infty_c(B)$ is nonnegative.  There is a constant $C$ depending only on the list $\Pi$ and $|\zeta|_{W^{2,\infty}(B)}$ such 
that 
$$
\zeta(x) |Du^\epsilon(x)|\le C, \quad x\in B.  
$$
\end{lem}
\begin{proof} 1. Set
$$
M_\epsilon:=\sup_{x\in B}|\zeta(x)Du^\epsilon(x)|
$$
and define
$$
v^\epsilon(x):=\frac{1}{2}\zeta^2(x)|Du^\epsilon(x)|^2 -  \alpha_\epsilon u^\epsilon(x).
$$
Here $\alpha_\epsilon$ is a positive constant that will be chosen below. We will first obtain a bound on $v^\epsilon$ from above and then 
use the resulting estimate to bound $M_\epsilon$.  We emphasize that each constant below will only depend on the list $\Pi$ and $|\zeta|_{W^{2,\infty}(B)}$; in particular, the constants will not depend on $\epsilon$ and $\alpha_\epsilon$. 
 
\par 2. We first differentiate equation \eqref{PenalizedEqn} with respect to $x_k$ $(k=1\dots, n)$ to get  
\begin{equation}\label{1stDerPenEqn}
\delta u^\epsilon_{x_k}+F^\varrho_{M_{ij}}(D^2u^\epsilon)u^\epsilon_{x_i x_j x_k} + \beta'_\epsilon(H^\varrho(Du^\epsilon))DH^\varrho(Du^\epsilon)\cdot Du^\epsilon_{x_k}=f^\varrho_{x_k}. 
\end{equation}
We suppress $\epsilon, \varrho$ dependence and function arguments and use \eqref{1stDerPenEqn} to compute
\begin{align}\label{BernIdentity1}
F_{M_{ij}}v_{x_i x_j}+ \beta' H_{p_k}v_{x_k}&=\left(F_{M_{ij}}\zeta_{x_i}\zeta_{x_j} + \zeta F_{M_{ij}}\zeta_{x_i x_j}\right)|Du|^2  + \nonumber \\
& \quad\quad 4F_{M_{ij}}\zeta\zeta_{x_i} Du\cdot Du_{x_j}  +\zeta^2 F_{M_{ij}}Du_{x_i}\cdot Du_{x_j}\nonumber \\
&\quad\quad - \beta' H_{p_k}(\alpha u_{x_k}- \zeta\zeta_{x_k}|Du|^2) \nonumber \\
& \quad\quad +\zeta^2 u_{x_k}(f_{x_k}-\delta u_{x_k}) - \alpha F_{M_{ij}}u_{x_i x_j}.
\end{align}
We reiterate that in \eqref{BernIdentity1}, we have written
$u$ for $u^\epsilon$, $v$ for $v^\epsilon$, $F$ for $F^\varrho(D^2u^\epsilon)$, $\beta$ for $\beta_\epsilon(H^\varrho(Du^\epsilon))$, $H$ for $H^\varrho(Du^\epsilon)$ and $f$ for $f^\varrho$.  We will continue this convention for the remainder of this proof. 

\par 3. Now we recall Lemma \ref{HomogeneityLEM}. In particular, the inequality \eqref{FrhoAlmostHomo} along with the convexity of $\beta=\beta_\epsilon$ implies
\begin{align*}
-F_{M_{ij}}u_{x_i x_j} &:=-F_{M_{ij}}(D^2u)(D^2u)_{ij}\\
& \le -F(D^2u)+\sqrt{n}\Theta\varrho_1\\
&= \beta(H(Du))+\delta u -f+\sqrt{n}\Theta\varrho_1\\
&\le H(Du)\beta'(H(Du))+\delta u -f+\sqrt{n}\Theta\varrho_1.
\end{align*}
Combining with \eqref{BernIdentity1} gives
\begin{align}\label{BernIdentity2}
F_{M_{ij}}v_{x_i x_j}+ \beta' H_{p_k}v_{x_k}&\le \left(F_{M_{ij}}\zeta_{x_i}\zeta_{x_j} + \zeta F_{M_{ij}}\zeta_{x_i x_j}\right)|Du|^2  + \nonumber \\
& \quad\quad 4F_{M_{ij}}\zeta\zeta_{x_i} Du\cdot Du_{x_j}  +\zeta^2 F_{M_{ij}}Du_{x_i}\cdot Du_{x_j}\nonumber \\
&\quad\quad - \beta' (\alpha(H_{p_k}u_{x_k}-H)-\zeta H_{p_k}\zeta_{x_k}|Du|^2) \nonumber \\
& \quad\quad +\zeta^2 u_{x_k}(f_{x_k}-\delta u_{x_k}) +\alpha(\delta u -f+\sqrt{n}\Theta\varrho_1).
\end{align}
\par 3. Assume $x_0\in \overline{B}$ is a maximizing point for $v$. If $x_0\in \partial B$, then $v\le -\alpha u_{\delta,\varrho}(x_0)\le -\alpha \inf_{\R^n} f$. Therefore, 
\begin{equation}\label{v1Upp}
v\le   C (\alpha+1).
\end{equation}
Alternatively, suppose $x_0\in B.$ If $\beta'=\beta'(H(Du(x_0)))\le 1<1/\epsilon$, then $H(Du(x_0))\le 2\epsilon\le 2$. By \eqref{Coercive}, 
$|Du(x_0)|$ is bounded from above independently of $\epsilon$. Hence, the \eqref{v1Upp} holds for an appropriate constant $C$.  The final situation to consider is when $\beta'=\beta'(H(Du(x_0)))>1$.   

\par Recall the uniform ellipticity assumption gives
$$
\eta^2 F_{M_{ij}}Du_{x_i}\cdot Du_{x_j}\le -\zeta^2\theta |D^2u|^2.
$$
And employing necessary conditions $Dv(x_0)=0$ and $D^2v(x_0)\le 0$ and the Cauchy-Schwarz inequality to the term 
$4F_{M_{ij}}\eta\eta_{x_i} Du\cdot Du_{x_j}\le (\zeta|D^2u|) (C|D\zeta||Du|)$ allow us to evaluate  \eqref{BernIdentity2} at
the point $x_0$ to get 
\begin{align*}
0 & \le C(|Du|^2+1+\alpha) - \beta' (\alpha(H_{p_k}u_{x_k}-H)-\zeta H_{p_k}\zeta_{x_k}|Du|^2)\\
& \le C(|Du|^2+1+\alpha) - \beta' (\sigma\alpha |Du|^2- C_0 (1+\zeta|Du|)|Du|^2)\\
& \le C\beta'\left\{|Du|^2+1+\alpha - \sigma\alpha |Du|^2+C_0 (1+\zeta|Du|)|Du|^2\right\}.
\end{align*}
 After multiplying through by $\zeta=\zeta(x_0)^2$ we have 
\begin{equation}\label{betaprimeINeq}
0\le C\beta'\left\{(\zeta|Du|)^2+1+\alpha - \sigma\alpha (\zeta |Du|)^2+C_0 (1+\zeta|Du|)(\zeta |Du|)^2\right\}
\end{equation}
which of course holds at $x_0$.

\par We now choose
$$
\alpha:=\frac{2C_0}{\sigma}M_\epsilon.
$$
Note $\sigma\alpha \ge 2C_0 \zeta(x_0)|Du(x_0)|$ and so \eqref{betaprimeINeq} gives
$$
0\le C\beta'\left\{(\zeta|Du|)^2+1+\alpha - 2C_0 (\zeta |Du|)^3+C_0 (1+\zeta|Du|)(\zeta |Du|)^2\right\}.
$$
As $\beta'>1$, the expression in the parentheses is necessarily nonnegative. It follows that there is constant $C$ such that 
$$
\zeta(x_0) |Du(x_0)|\le C(1+\alpha)^{1/3}. 
$$
As a result, \eqref{v1Upp} holds for another appropriately chosen constant $C$.  
\par 4. Therefore, 
$$
M_\epsilon^2=\sup_{B}|\zeta Du^\epsilon|^2 =2\sup_{B}(v^\epsilon + \alpha_\epsilon u^\epsilon)\le C(\alpha_\epsilon+1)
\le C\left(\frac{2C_0}{\sigma}M_\epsilon+1\right). 
$$
Consequently, $M_\epsilon$ is bounded above independently of $\epsilon\in (0,1)$. 
\end{proof}
Next we assert that $\beta_\epsilon(H^\varrho(Du^\epsilon))$ is locally bounded, independently of all $\epsilon$ sufficiently small. 
\begin{lem}\label{betaboundLem}
Let $\delta\in(0,1)$, $\varrho\in (0,\varrho_1)$, $\epsilon\in (0,1)$ and suppose $\zeta\in C^\infty_c(B)$ is nonnegative.  There is a constant $C$ depending only on the list $\Pi$ and $|\zeta|_{W^{2,\infty}(B)}$ such 
that 
$$
\zeta(x)\beta_\epsilon(H^\varrho(Du^\epsilon(x)))\le C, \quad x\in B.  
$$
\end{lem}
We omit a proof of Lemma \ref{betaboundLem} as the proof of Lemma 3.3 in our recent work \cite{HyndMawi} immediately applies here.  We also note that 
$$
F^\varrho(D^2u^\epsilon)=-\beta_\epsilon(H^\varrho(Du^\epsilon)) +f^\varrho-\delta u^\epsilon
$$
is locally bounded, independently of $\epsilon \in (0,1]$. By the $W^{2,p}_{\text{loc}}$ estimates for fully nonlinear elliptic equations due to L. Caffarelli (Theorem 1 
in \cite{CaffAnn}, Theorem 7.1 in \cite{CC}), we have the following. 
\begin{lem}\label{W2pBoundUeps}
Let $\delta\in(0,1)$, $\varrho\in (0,\varrho_1)$, $\epsilon\in (0,1)$, $p\in (n,\infty)$, and assume $G\subset B$ is open with $\overline{G}\subset B$. There is a constant $C$ depending on $p$, the list $\Pi$, $1/\dist(\partial G,B)$ and $G$ such that 
$$
|D^2u^\epsilon|_{L^p(G)}\le C\left\{|u^\epsilon|_{L^\infty(B)}+1\right\}. 
$$
\end{lem}
\begin{proof}
Assume $B_r(x_0)\subset B$ is nonempty, and choose $\zeta\in C^\I_c(B_r(x_0))$ such that $0\le \zeta\le 1$, $\zeta\equiv 1$ on $B_{r/2}(x_0)$ and 
\begin{equation}\label{DzetaBounds}
|D\zeta|_{L^\I(B_{r/2}(x_0))}\le \frac{C}{r}, \quad  |D^2\zeta|_{L^\I(B_{r/2}(x_0))}\le \frac{C}{r^2}.
\end{equation}
From Lemma \ref{betaboundLem},  $\beta_\epsilon(H^\varrho(Du^\epsilon(x)))\le C_1$ for $x\in B_{r/2}(x_0)$ for some $C_1$ depending only on the list $\Pi$ and $r$. By the assumption that $F$ is uniformly elliptic and concave, Theorem 7.1 in \cite{CC} implies there is a universal constant $c_0$ such 
that
\begin{align*}
r^2|D^2u^\epsilon|_{L^p(B_{r/4}(x_0))}&\le c_0\left\{|u^\epsilon|_{L^\I(B_{r/2}(x_0))} + |f^\varrho-\delta u^\epsilon - \beta_\epsilon(H^\varrho(Du^\epsilon))|_{L^\I(B_{r/2}(x_0))}\right\}\\
&\le  c_0\left\{|u^\epsilon|_{L^\I(B_{r/2}(x_0))} + |f^\varrho-\delta u^\epsilon|_{L^\I(B_{r/2}(x_0))}+C_1\right\}\\
&\le C_0\left\{|u^\epsilon|_{L^\infty(B)} + |f^\varrho|_{L^\I(B)}+C_1 \right\}.
\end{align*}
Here $C_0$ only depends only on the list $\Pi$.

\par Now select $r=\frac{1}{2}\dist(\partial G, B)$ and cover $\overline{G}$ with finitely many balls $B_{r/4}(x_1), \dots, B_{r/4}(x_m)$, with each $x_1,\dots,x_m\in G$. Then 
\begin{align*}
\int_G|D^2u^\epsilon(x)|^pdx&\le \int_{\cup^m_{i=1}B_{r/4}(x_i)}|D^2u^\epsilon(x)|^pdx \\
&\le \sum^m_{i=1}\int_{B_{r/4}(x_i)}|D^2u^\epsilon(x)|^pdx \\
&\le mC_0^p\left(|u^\epsilon|_{L^\infty(B)} + |f^\varrho|_{L^\I(B)}+C_1\right)^p.
\end{align*}
\end{proof}
In view of our uniform estimates, we are in position to send $\epsilon\rightarrow 0^+$ in the equation \eqref{PenalizedEqn}. 
\begin{prop}
Let $\delta\in (0,1)$, $\varrho\in (0,\varrho_1)$, $p\in (n,\infty)$ and assume $G\subset B$ is open with $\overline{G}\subset B$. \\
(i) There is $v_{\delta,\varrho}\in C(\overline{B})\cap W^{2,p}_{\text{loc}}(B)$ such that
$u^\epsilon\rightarrow v_{\delta,\varrho}$, as $\epsilon\rightarrow 0^+$, uniformly in $\overline{B}$ and weakly in $W^{2,p}(G)$.  \\
(ii) Moreover, $v_{\delta,\varrho}$ is the unique solution of the boundary value problem
\begin{equation}\label{vdeltavarrhoPDE}
\begin{cases}
\max\{\delta v+F^\varrho(D^2v)-f^\varrho,H^\varrho(Dv)\}=0&\quad x\in B \\
\hspace{2.38in}v=u_{\delta,\varrho}& \quad x\in \partial B
\end{cases}.
\end{equation}
(iii) There is a constant $C$ depending on $p$, the list $\Pi$, $1/\dist(\partial G,B)$ and $G$ such that 
\begin{equation}\label{vdeltarhoEst1}
|D^2v_{\delta,\varrho}|_{L^p(G)}\le C\left\{|v_{\delta,\varrho}|_{L^\I(B)} +1\right\}
\end{equation}
and 
\begin{equation}\label{vdeltarhoEst2}
-C\le F^\varrho(D^2v_{\delta,\varrho}(x))
\end{equation}
for Lebesgue almost every $x\in G$.
\end{prop}
\begin{proof}
$(i)-(ii)$ The convergence to $v$ satisfying \eqref{vdeltavarrhoPDE} is proved very similar to Proposition 4.1 in \cite{Hynd} and part $(ii)$ of Theorem 1.1 in \cite{HyndMawi}, so we omit the details.  In both arguments, the uniqueness of solutions of a related boundary value problem of the type \eqref{vdeltavarrhoPDE} is crucial; in our case, uniqueness follows from the estimate \eqref{GenComparisonEst} below. 

\par $(iii)$  The bound \eqref{vdeltarhoEst1} follows from part $(i)$ and Lemma \ref{W2pBoundUeps}. Let us now verify   \eqref{vdeltarhoEst2}.   Recall that $F^\varrho$ is concave. As $u^\epsilon$ converges to $v_{\delta,\varrho}$ weakly $W^{2,p}_{\text{loc}}(B)$, for each $\zeta\in C^\I_c(B)$ that is nonnegative,
\begin{equation}\label{LimSupConcav}
\limsup_{\epsilon\rightarrow 0^+}\int_B F^\varrho(D^2u^\epsilon(x))\zeta(x)dx\le \int_B F^\varrho(D^2v_{\delta,\varrho}(x))\zeta(x)dx. 
\end{equation}
By Lemma \ref{betaboundLem}, there is a constant $C$ depending only the list $\Pi$ and $|\zeta|_{W^{2,\infty}(B)}$ such 
that $\zeta F^\varrho(D^2u^\epsilon)\ge -C$. Inequality \eqref{LimSupConcav} then gives 
\begin{equation}\label{vdeltarhoEst3}
-C\le \zeta(x) F^\varrho(D^2v_{\delta,\varrho}(x))
\end{equation}
for almost every $x\in B$.  Let $x_0\in G$ and $r:=\frac{1}{2}\dist(\partial G, B)$, and choose $0\le \zeta\le 1$ to be supported in $B_r(x_0)$ and satisfy $\zeta\equiv 1$ on $B_{r/2}(x_0)$ and \eqref{DzetaBounds}. Then \eqref{vdeltarhoEst3} implies that \eqref{vdeltarhoEst2} holds for almost every $x\in B_{r/2}(x_0)$ for some constant $C$ depending on $\Pi$ and $r$. The general
bound follows by a routine covering argument.
\end{proof}

\begin{prop}\label{FFvarrholocLem} Let $\delta\in(0,1)$, $p\in (n,\infty)$ and assume $G\subset B$ is open with $\overline{G}\subset B$.\\ 
$(i)$ Then $v_{\delta,\varrho}\rightarrow u_\delta$, as $\varrho\rightarrow 0^+$, uniformly on $\overline{B}$ and weakly in $W^{2,p}(G)$. 
\\ 
$(ii)$ There is a constant $C$ depending on $p$, the list $\Pi$, $1/\dist(\partial G,B)$ and $G$ such that 
$$
-C\le F(D^2u_{\delta}(x))
$$
for almost every $x\in G$.
\end{prop} 
\begin{proof}
$(i)$ We first claim
\begin{equation}\label{UdeltarhoVdeltarhoEst}
u_{\delta,\varrho}(x)\le v_{\delta,\varrho}(x)\le u_{\delta,\varrho}(x) +\frac{1}{\delta}\sqrt{n}\Theta\varrho
\end{equation}
for $x\in B$ and $\varrho\in (0,\varrho_1)$.  And in order to prove \eqref{UdeltarhoVdeltarhoEst}, we will need the estimate 
\begin{equation}\label{GenComparisonEst}
\max_{\overline{B}}\{u-v\}\le \max_{\partial B}\{u-v\}+\frac{1}{\delta}\max_{\overline{B}}\{g-h\}
\end{equation}
which holds for each $u\in USC(\overline{B})$ and $v\in LSC(\overline{B})$ that satisfy 
\begin{equation}\label{GenComparisonPDE}
\max\{\delta u+F(D^2u)-g,H^\varrho(Du)\}\le 0\le \max\{\delta v+F(D^2v)-h,H^\varrho(Dv)\}, \quad x\in B. 
\end{equation}
Here $g,h\in C(\overline{B})$. The estimate \eqref{GenComparisonEst} can be proved with the ideas 
used to verify Proposition \eqref{lamCompProp}; see also Proposition 2.2 of \cite{HyndMawi}.  We leave the details to the reader. 

\par Using $F^\varrho\le F$, the inequality $u_{\delta,\varrho}\le v_{\delta,\varrho}$ follows from \eqref{GenComparisonEst} as $u=u_{\delta,\varrho}$, $v=v_{\delta,\varrho}$ satisfy \eqref{GenComparisonPDE} $g=h=f^\varrho$.  Likewise, we can use the bound 
$F^\varrho+\sqrt{n}\Theta \varrho$ to show the inequality $v_{\delta,\varrho}\le u_{\delta,\varrho} + \sqrt{n}\Theta \varrho/\delta$ follows from \eqref{GenComparisonEst} as $u=v_{\delta,\varrho}$, $v=u_{\delta,\varrho}$ satisfy \eqref{GenComparisonPDE} with $g=f^\varrho+\frac{1}{\delta}\sqrt{n}\Theta\varrho$ and $h=f^\varrho$.  The assertion that 
$v_{\delta,\varrho}$ converges to $u_\delta$ in $W^{2,p}(G)$ weakly follows from \eqref{vdeltarhoEst1}. 

\par $(ii)$ Let $U\subset G$ be measurable and recall that $F^\varrho\le F$ and $F$ is concave.  By \eqref{vdeltarhoEst2}, we have there is a constant $C$ depending on $p$, the list $\Pi$, $1/\dist(\partial G,B)$ and $G$ such that
\begin{align*}
-C|U|&\le \limsup_{\varrho \rightarrow 0^+}\int_{U}F^\varrho(D^2v_{\delta,\varrho}(x))dx \\
&\le \limsup_{\varrho \rightarrow 0^+}\int_{U}F(D^2v_{\delta,\varrho}(x))dx \\
&\le \int_{U}F(D^2u_{\delta}(x))dx.
\end{align*} 
\end{proof}

\begin{cor}\label{W2infuDel} For each $\delta\in (0,1)$, $D^2u_\delta\in L^\I(\R^n;\Sn)$. Moreover, there is a constant $C$ depending only on 
the list $\Pi$ for which
$$
|D^2u_\delta|_{ L^\I(\R^n;\Sn)}\le C.
$$
for each $\delta\in(0,1)$.
\end{cor}
\begin{proof}
Choose $R_1>0$ so that $\Omega_\delta\subset B_{R_1}(0)$ for all $\delta\in (0,1)$; such an $R_1$ exists by corollary \ref{boundedDerOmega}.  Lemma \ref{W2infFarOutBound} gives that there is a universal constant $C$ such that  
$$
D^2u_\delta(x)\le \frac{C}{R_1}I_n. 
$$
for almost every $|x|\ge 2R_1$. 

\par Now select $R>2R_1$ so large that \eqref{BbigEnough} is satisfied. Part $(ii)$ of Proposition \ref{FFvarrholocLem}, with $G=B_{2R_1}(0)$ and $B=B_R(0)$, gives a constant $C_1$ depending on $R_1$ and the list $\Pi$ such that  
\begin{equation}\label{FD2uBounds}
-C_1\le F(D^2u_\delta(x))
\end{equation}
for almost every $|x|\le 2R_1$. Since $u_\delta$ is convex (Proposition \ref{UdelConvex}), the uniform ellipticity assumption on $F$ implies 
\begin{equation}\label{FD2uBounds2}
F(D^2u_\delta(x))\le -\theta\Delta u_\delta(x)
\end{equation}
for almost every $x\in \R^n$. Therefore, we can again appeal to the convexity of $u_\delta$ and employ \eqref{FD2uBounds} and \eqref{FD2uBounds2} 
to get 
$$
D^2u_\delta(x)\le \Delta u_\delta(x) I_n\le \frac{C_1}{\theta}I_n
$$
for almost every $|x|\le 2R_1$. 
\end{proof}
\begin{proof} (part $(ii)$ of Theorem \ref{Thm1})
By the convexity of $u_\delta$ and Corollary \ref{W2infuDel}, there is a constant $C$ independent of $\delta\in (0,1)$ for which 
$$
0\le u_\delta(x+h)-2u_\delta(x)+u_\delta(x-h)\le C|h|^2
$$
for every $x, h\in \R^n$. The assertion now follows from passing to the limit along an appropriate sequence $\delta$ tending to $0$ 
as was done in the proof of part $(i)$ of Theorem \ref{Thm1}.
\end{proof}
\begin{rem}\label{remboundedOmega}
By part $(ii)$ of Theorem \ref{Thm1}, $Du_\delta$ exists everywhere and is continuous. By Corollary \ref{boundedDerOmega}  
$$
\Omega_\delta=\{x\in \R^n : H(Du_\delta(x))<0\}
$$
is open and bounded. 
\end{rem}

\section{1D and rotationally symmetric problems}\label{1DandRotSymmSect}
Now we will discuss a few results for solutions of the eigenvalue problem \eqref{EigProb} when the dimension 
$n=1$ and when $F,f,H$ satisfy the symmetry hypothesis  \eqref{SymmetryCond}: 
$$
\begin{cases}
f(Ox)=f(x)\\
H(O^tp)=H(p)\\
F(OMO^t)=F(M)
\end{cases}
$$
for each $x,p\in \R^n$, $M\in \Sn$ and orthogonal $n\times n$ matrix $O$.  First, we prove Theorem \ref{SymmRegThm} which 
involves the regularity of symmetric eigenfunctions. Then we consider the uniqueness of eigenfunctions of \eqref{EigProb} that 
satisfy the growth condition \eqref{ellgrowth}.

\begin{proof} (Theorem \ref{SymmRegThm}) The assumption \eqref{SymmetryCond} implies that $u_\delta$ is radial; this follows from the uniqueness assertion \ref{DeltaCompProp}.  In particular, $u^*$ constructed in the proof of part $(i)$ of Theorem \ref{Thm1} will also be radial. Consequently, there is a function $\phi: [0,\infty)\rightarrow \R$ such that $u^*(x)=\phi(|x|)$. As $u^*$ is convex, $\phi$ is nondecreasing 
and convex. Moreover, for almost every $x\in \R^n$
$$
\begin{cases}
Du^*(x)=\phi'(|x|)\frac{x}{|x|}\\
D^2u^*(x)=\phi''(|x|)\frac{x\otimes x}{|x|^2} +\frac{\phi'(|x|)}{|x|}\left(I_n - \frac{x\otimes x}{|x|^2}\right)
\end{cases}.
$$
\par Similar arguments imply $f(x)=f_0(|x|)$ for a nondecreasing, convex function $f_0$. Likewise $H(p)$ only depends on $|p|$ and so $\{p\in \R^n: H(p)\le 0\}$ is a ball. Thus, $\ell(v)=a|v|$ for some 
$a>0$, and as a result $H_0(p)=|p|-a$.  The assumption \eqref{SymmetryCond} also implies  $F=F(M)$ only depends on the eigenvalues of $M$. 
In particular, the symmetric function $G(\mu_1,\dots,\mu_n):=F(\diag(\mu_1,\dots,\mu_n))$ completely determines $F$.  And as $F$ is uniformly elliptic 
$$
G(\mu_1+h,\dots,\mu_n)-G(\mu_1,\dots,\mu_n)\le-\theta h.
$$
for $h\ge 0$. 

\par From our comments above, $\phi$ satisfies
\begin{equation}\label{phiEq}
\max\left\{\lambda^*+G\left(\frac{\phi'}{r},\dots,\frac{\phi'}{r}, \phi''\right) - f_0(r), \phi'-a\right\}=0, \quad r>0.
\end{equation}
And since $\phi'$ is nondecreasing, 
$$
\{r>0: \phi'(r)<a\}=(0,r_0)
$$
for some $r_0>0$; this is another way of expressing $\Omega:=\{x\in \R^n: H(Du^*(x))<0\}=B_{r_0}(0)$.  Part $(ii)$ of Theorem \ref{Thm1} then implies $u^*\in C^2(\Omega)\cap C^{1,1}_\text{loc}(\R^n)$.  Thus, $\phi'=a$ for $r\ge a$ and $\phi\in C^2(\R\setminus\{r_0\})\cap C^{1,1}_\text{loc}(\R)$. Furthermore, as $\phi''(r_0+)=0$, we just need to show $\phi''(x_0-)=0$. 

\par Recall the left hand limit $\phi''(x_0-)$ exists and is nonnegative since $\phi$ is convex.  By \eqref{phiEq}, 
$$
\lambda^*+G\left(\frac{\phi'}{r},\dots,\frac{\phi'}{r}, \phi''\right) - f_0(r)\le 0, \quad r>0.
$$
Sending $r\rightarrow r_0^+$ gives
\begin{equation}\label{Geq1}
\lambda^*+G\left(\frac{a}{r_0},\dots,\frac{a}{r_0}, 0\right) - f_0(r_0)\le 0.
\end{equation}
Now, 
$$
\lambda^*+G\left(\frac{\phi'}{r},\dots,\frac{\phi'}{r}, \phi''\right) - f_0(r)= 0, \quad r\in (0,r_0)
$$
and sending $r\rightarrow r_0^-$ gives 
\begin{equation}\label{Geq2}
\lambda^*+G\left(\frac{a}{r_0},\dots,\frac{a}{r_0},\phi''(x_0-) \right) - f_0(r_0)=0.
\end{equation}
Combining \eqref{Geq1} and \eqref{Geq2} gives 
$$
G\left(\frac{a}{r_0},\dots,\frac{a}{r_0}, 0\right) \le f_0(r_0)-\lambda^*=G\left(\frac{a}{r_0},\dots,\frac{a}{r_0},\phi''(x_0-) \right).
$$
By the monotonicity of $G$ in each of its arguments, $\phi''(x_0-)\le 0$. Thus $\phi''(x_0)=0$, and as a result, $u^*\in C^2(\R^n)$.  
\end{proof}

\begin{prop}\label{UniqunessN1}
Assume $n=1$. Any two convex solutions of \eqref{EigProb} that satisfy \eqref{ellgrowth} differ by an additive constant. 
\end{prop}
\begin{proof}
Assume $u_1, u_2$ are convex and satisfy 
$$
\begin{cases}
\max\{\lambda^*+F(u'')-f,H(u')\}=0, \quad x\in \R\\
\lim\frac{u(x)}{\ell(x)}=1
\end{cases}.
$$
As in the proof of Theorem \ref{SymmRegThm}, we may deduce that necessarily $u_1,u_2\in C^2(\R)$. Also observe
$$
H_0(p)=\max_{v\pm 1}\left\{pv-\ell(v)\right\}=\max\{p-\ell(1), -p -\ell(-1)\}.
$$
In particular, 
$$
\{p\in \R: H(p)\le 0\}=[-\ell(-1), \ell(1)].
$$
It then follows from the convexity of $u_1$ and $u_2$ that 
$$
I_1:=\{x\in \R: H(u_1'(x))<0\}\quad \text{and}\quad I_2:=\{x\in \R: H(u_2'(x))<0\}
$$
are bounded, open intervals. 

\par Let us first assume $I_1=I_2=(\alpha,\beta)$. Then 
$$
\lambda^*+F(u_1'')-f=0=\lambda^*+F(u_2'')-f, \quad x\in  (\alpha,\beta)
$$
As $F$ is uniformly elliptic $u_1''=u_2''=F^{-1}(f-\lambda^*)$ for $x\in (\alpha,\beta)$. Hence, $u'_1-u_2'$ is constant. The above characterization of $\{p\in \R: H(p)\le 0\}$ 
also implies
$$
\begin{cases}
u_1'=u_2'=-\ell(-1), \quad x\in (-\infty,\alpha]\\
u_1'=u_2'=\ell(1), \quad x\in [\beta,\infty)
\end{cases}
$$
It now follows that necessarily $u'_1=u'_2$ and so $u_1-u_2$ is constant. 

\par Now we are left to prove that $I_1=I_2$; for definiteness, we shall assume $I_1= (\alpha_1,\beta_1)$ and $I_2= (\alpha_2,\beta_2)$. First suppose that $I_1\cap I_2=\emptyset$ and without loss of generality  $\beta_1<\alpha_2$. 
Then on $I_1$, $\lambda^*+F(u_1'')-f=0$ and $u_2'=-\ell(-1)$. We always have $\lambda^*+F(u_2'')-f\le 0$ which implies $\lambda^*-f\le 0$ on $I_1$ since $u_2''=0$. It then 
follows that $F(u_1'')=f-\lambda^*\ge 0$ and thus $u_1''\le 0.$ As $u$ is convex, $u_1''=0$ in $I_1.$  However, $u_1'$ is constant and it would then be impossible for $u'_1(\alpha_1)=-\ell(-1)<0$ and 
$u'_1(\beta_1)=\ell(1)>0$.  Therefore, $I_1\cap I_2\neq\emptyset$.

\par Without any loss of generality, we may assume $\alpha_1<\alpha_2<\beta_1$. Repeating our argument above, we find $u_1''=0$ on $(\alpha_1,\alpha_2)$.  It must be that $u_1'$ is constant 
and thus equal to $-\ell(-1)$ on $[\alpha_1,\alpha_2]$. But then $H(u_1')=0$ on $[\alpha_1,\alpha_2]$, which contradicts the definition of $I_1$. Hence, $I_1=I_2$ and the assertion follows. 
\end{proof}

\begin{prop}\label{ConvUniqueness}
Assume the symmetry condition \eqref{SymmetryCond} and that $F$ is uniformly elliptic. Then any two convex, rotationally symmetric solutions of \eqref{EigProb} that satisfy
\eqref{ellgrowth} differ by an additive constant. 
\end{prop}

\begin{proof}
As remarked in the above proof of Theorem \ref{SymmRegThm}, the symmetry assumption on $H$ results in $H_0(p)=|p|-a$ for some $a>0$. Now assume $u_1, u_2$ are convex, rotationally symmetric solutions of \eqref{EigProb} that satisfy \eqref{ellgrowth}. Then it follows
$$
\{x\in\R^n: H(Du_i(x))<0\}=B_{r_i}(0)
$$
for $i=1,2$ and some $r_1,r_2>0$.  Thus, 
\begin{equation}\label{u1u2const}
u_i(x)=a|x|+b_i, \quad |x|\ge r_i
\end{equation}
for some constants $b_i$. If $r_1=r_2=:r$, then 
$$
\begin{cases}
F(D^2u_1)=f(x)-\lambda^*=F(D^2u_2), \quad &x\in B_{r}(0)\\
u_1=ar+b_1, \quad u_2=ar+b_2, \quad & x\in \partial B_r(0).
\end{cases}
$$
As $F$ is uniformly elliptic, $u_1\equiv u_2+b_1-b_2$ on $\overline{B_r(0)}$ and thus on $\R^n$.

\par Now suppose $r_1<r_2$. And set $v:= u_2+b_1-b_2$; from \eqref{u1u2const} $u_1\equiv v$ for $|x|\ge r_2$. Since, 
$$
F(D^2u_1)\le f(x)-\lambda^*=F(D^2v),\quad x\in B_{r_2}
$$
the maximum principle implies $u_1\le v$ in $\overline{B}_{r_2}$. The strong maximum principle implies $u_1\equiv v$ in $\overline{B}_{r_2}$, from 
which we conclude the proof, or  $u_1< v$ in $B_{r_2}$. However if $u_1< v$ in $B_{r_2}$, Hopf's Lemma (see the appendix of \cite{Armstrong}) implies 
\begin{equation}\label{HopfCond}
\frac{\partial v}{\partial \nu}(x_0)<\frac{\partial u_1}{\partial \nu}(x_0)
\end{equation}
for each $x_0\in\partial B_{r_2}$. Here $\nu=x_0/|x_0|$.  As $u$ and $v$ are rotational and convex, \eqref{HopfCond} implies
$$
|Dv(x_0)|<|Du_1(x_0)|\le a.
$$
However, $|Dv|=a$ on $\partial B_{r_2}$. This contradicts the hypothesis that $r_1<r_2$.  
 \end{proof}

\section{Minmax formulae}\label{MinMaxSect}
This final section is devoted entirely to the proof of Theorem \ref{minmaxThm}.  In particular, we will make use of the characterizations of $\lambda^*$ given in \eqref{LamChar1} and \eqref{LamChar2}.  
We will also use that the functions $H$ and $H_0$ have the same sign. 

\par  Let $\phi \in C^2(\R^n)$ and suppose that $H(D\phi)\le 0$. If 
$$
\lambda_\phi:=\inf_{\R^n}\left\{-F(D^2\phi(x))+f(x)\right\}>-\infty,
$$ 
then $\phi$ is a subsolution of \eqref{EigProb} with eigenvalue $\lambda_\phi$. By \eqref{LamChar1}, $\lambda_\phi\le \lambda^*$. Hence, $\lambda_-=\sup_\phi\lambda_\phi\le \lambda^*.$  Now let $\psi\in C^2(\R^n)$ satisfy 
$$
\liminf_{|x|\rightarrow \infty}\frac{\psi(x)}{\ell(x)}\ge 1.
$$
If 
$$
\mu_\psi:=\sup_{H(D\psi)<0}\left\{-F(D^2\psi(x))+f(x)\right\}<\infty,
$$ 
then $\psi$ is a supersolution of \eqref{EigProb} with eigenvalue $\mu_\psi$. It follows from  \eqref{LamChar2} that $\lambda_\psi\ge \lambda^*$.  As a result, $\lambda_+=\inf_\psi\mu_\psi\ge \lambda^*.$

\par Let $u^*$ be an eigenfunction associated with $\lambda^*$ that satisfies $D^2u^*\in L^\infty(\R^n; S_n(\R))$. As in Remark \ref{remboundedOmega}, 
$$
\Omega_0:=\{x\in \R^n: H(Du^*(x))<0\}
$$
is open and bounded.  For $\epsilon>0$ and 
$\tau>1$, set 
$$
u^{\epsilon,\tau}=\tau u^\epsilon=\tau(\eta^\epsilon*u^*).
$$
Here $u^\epsilon=\eta^\epsilon*u^*$ is the standard mollification of $u^*$. Observe that $H_0$ 
is Lipschitz continuous with Lipschitz constant no more than one; in view of the basic estimate $|Du^*-Du^\epsilon|_{L^\infty(\R^n)}\le \epsilon |D^2u^*|_{L^\infty(\R^n)}$,
\begin{equation}\label{LipHCond}
H_0(Du^*(x))\le H_0(Du^\epsilon(x))+\epsilon |D^2u^*|_{L^\infty(\R^n)}, \quad x\in \R^n.
\end{equation}

\par So for any $x\in \R^n$ where $H(Du^{\epsilon,\tau}(x))<0$, 
\begin{equation}\label{LipHCond2}
H_0(Du^{\epsilon}(x))=H_0\left(\frac{1}{\tau}Du^{\epsilon,\tau}(x)+\frac{\tau-1}{\tau}0\right)<
\frac{\tau-1}{\tau}H_0(0)<0.
\end{equation}
In view of \eqref{LipHCond} and \eqref{LipHCond2}, we can choose $\epsilon_1=\epsilon_1(\tau)>0$ such that 
$$
\varrho:=-\left(\frac{\tau-1}{\tau}H_0(0)+\epsilon_1 |D^2u^*|_{L^\infty(\R^n)}\right)>0
$$
and
$$
\{x\in \R^n: H(Du^{\epsilon,\tau}(x))<0\}\subset\{x\in \R^n: H_0(Du^*(x))<-\varrho\}
$$
for $\epsilon\in (0,\epsilon_1)$.  Since $\{x\in \R^n: H_0(Du^*(x))<-\varrho\}$ is a proper open subset 
of $\Omega_0$ we can further select $\epsilon_2=\epsilon_2(\tau)>0$ so that 
\begin{equation}\label{ImportantIncludsionepstau}
\{x\in \R^n: H_0(Du^*(x))<-\varrho\}\subset \Omega^\epsilon:=
\{x\in \R^n: \text{dist}(x,\partial\Omega_0)>\epsilon\}
\end{equation}
for $\epsilon\in (0,\epsilon_2)$.

\par By assumption, $u^*$ satisfies $\lambda^* +F(D^2u^*)-f=0$ for almost every $x\in \Omega_0$. Mollifying both sides of this equation gives $
\lambda^*+F(D^2u^*)^\epsilon-f^\epsilon=0$ in $\Omega^\epsilon$.
Since $F$ is concave 
$$
F(D^2u^\epsilon(x))=F\left(\int_{\R^n}\eta^\epsilon(y)D^2u^*(x-y)dy\right)\ge \int_{\R^n}\eta^\epsilon(y)F(D^2u^*(x-y))dy=F(D^2u^*)^\epsilon(x).
$$
Consequently, $\lambda^*+F(D^2u^\epsilon)-f^\epsilon\ge 0,$ in $\Omega^\epsilon$. And since 
$\Omega_0$ is bounded, $|f^\epsilon-f|_{L^\infty(\Omega_0)}=o(1)$ as $\epsilon \rightarrow 0^+$. Therefore,
\begin{equation}\label{lambdaMollifiedEqn}
\lambda^*+F(D^2u^\epsilon)-f\ge o(1), \quad x\in \Omega^\epsilon.
\end{equation}
as $\epsilon\rightarrow 0^+$. 

\par We can now combine the inclusion \eqref{ImportantIncludsionepstau} and the inequality \eqref{lambdaMollifiedEqn}. For $\epsilon\in (0,\min\{\epsilon_1,\epsilon_2\}$

\begin{align*}
\lambda^+&\le \sup_{H(Du^{\epsilon,\tau})<0}\left\{-F(D^2u^{\epsilon,\tau}(x)) +f(x)\right\}\\
&= \sup_{H(Du^{\epsilon,\tau})<0}\left\{-\tau F(D^2u^{\epsilon}(x)) +f(x)\right\}\\
&= \sup_{H(Du^{\epsilon,\tau})<0}\left\{- F(D^2u^{\epsilon}(x)) +f(x)\right\}+O(\tau-1)\\
&\le \sup_{\Omega^\epsilon}\left\{- F(D^2u^{\epsilon}(x)) +f(x)\right\}+O(\tau-1)\\
&\le \lambda^* +o(1)+O(\tau-1).
\end{align*}
We conclude by first ending $\epsilon\rightarrow 0^+$ and then $\tau\rightarrow 1^+$.



\begin{thebibliography}{}

\bibitem{Armstrong} Armstrong, S. N. \emph{Principal eigenvalues and an anti-maximum principle for homogeneous fully nonlinear elliptic equations}. J. Differential Equations 246 (2009), no. 7, 2958--2987.

\bibitem{Bardi} Bardi, M.; Capuzzo-Dolcetta, I. \emph{Optimal control and viscosity solutions of Hamilton-Jacobi-Bellman equations}. With appendices by Maurizio Falcone and Pierpaolo Soravia. Systems \& Control: Foundations \& Applications. Birkh\"{a}user Boston, Inc., Boston, MA, 1997. 

\bibitem{Borkar} Borkar, Vivek; Budhiraja, Amarjit. \emph{Ergodic control for constrained diffusions: characterization using HJB equations}. SIAM J. Control Optim. 43 (2004/05), no. 4, 1467--1492.

\bibitem{CaffAnn} Caffarelli, Luis A. \emph{Interior a priori estimates for solutions of fully nonlinear equations}. Ann. of Math. (2) 130 (1989), no. 1, 189--213.

\bibitem{CC} Caffarelli, Luis A.; Cabr\`{e}, Xavier. \emph{Fully nonlinear elliptic equations}. American Mathematical Society Colloquium Publications, 43. American Mathematical Society, Providence, RI, 1995.

\bibitem{Crandall} Crandall, M. \emph{Viscosity solutions: a primer}. Viscosity solutions and applications, 1--43, Lecture Notes in Math., 1660, Springer, Berlin, 1997.

\bibitem{CIL} Crandall, M.; Ishii, H.; Lions, P.-L. \emph{User's guide to viscosity solutions of second order partial differential equations}. Bull. Amer. Math. Soc. (N.S.) 27 (1992), no. 1, 1--67.

\bibitem{Evans} Evans, L. C. \emph{A second-order elliptic equation with gradient constraint}. Comm. Partial Differential Equations 4 (1979), no. 5, 555--572. 

\bibitem{EvansC2} Evans, L. C. \emph{Classical solutions of fully nonlinear, convex, second-order elliptic equations}. Comm. Pure Appl. Math. 35 (1982), no. 3, 333--363.

\bibitem{Evans2} Evans, L. C. \emph{Partial differential equations}. Second Edition. Graduate Studies in Mathematics, 19. American Mathematical Society, Providence, RI, 2010.

\bibitem{Gariepy} Evans, L. C.; Gariepy, R. \emph{Measure theory and fine properties of functions}. Studies in Advanced Mathematics. CRC Press, Boca Raton, FL, 1992.

\bibitem{Fleming} Fleming, W.; Soner, H. \emph{Controlled Markov processes and viscosity solutions}. Second edition. Stochastic Modeling and Applied Probability, 25. Springer, New York, 2006.

\bibitem{Gilbarg} Gilbarg, D.; Trudinger, N. \emph{Elliptic Partial Differential Equations of Second Order}. Springer (1998). 

\bibitem{Hynd} Hynd, R. \emph{The eigenvalue problem of singular ergodic control}. Comm. Pure Appl. Math. 65 (2012), no. 5, 649--682. 

\bibitem{HyndMawi} Hynd, R; Mawi, H. \emph{On Hamilton-Jacobi-Bellman equations with convex gradient constraints}. arxiv.org/abs/1412.6202.

\bibitem{Ishii} Ishii, H.; Koike, S. \emph{Boundary regularity and uniqueness for an elliptic equation with gradient constraint}. Comm. Partial Differential Equations 8 (1983), no. 4, 317--346.

\bibitem{KO}  Korevaar, N. J. \emph{Convex solutions to nonlinear elliptic and parabolic boundary value problems}. Indiana Univ. Math. J. 32 (1983), no. 4, 603--614.

\bibitem{Kruk} Kruk, L. \emph{Optimal policies for $n$-dimensional singular stochastic control problems}. II. The radially symmetric case. Ergodic control. SIAM J. Control Optim. 39 (2000), no. 2, 635--659.

\bibitem{KrylovC2} Krylov, N. V. \emph{Boundedly inhomogeneous elliptic and parabolic equations in a domain}. (Russian) Izv. Akad. Nauk SSSR Ser. Mat. 47 (1983), no. 1, 75--108.

\bibitem{Menaldi} Menaldi, J.-L.; Robin, M.; Taksar, M. \emph{Singular ergodic control for multidimensional Gaussian processes}. Math. Control Signals Systems 5 (1992), no. 1, 93--114. 

\bibitem{Oksendal} \O ksendal, B; Sulem, A. \emph{Applied stochastic control of jump diffusions}. Second edition. Universitext. Springer, Berlin, 2007.

\bibitem{Rock} Rockafellar, R. T.; Wets, R. \em{Variational analysis.} Grundlehren der Mathematischen Wissenschaften [Fundamental Principles of Mathematical Sciences], 317. Springer-Verlag, Berlin, 1998.

\bibitem{Soner} Soner, H. M.; Shreve, S. \emph{Regularity of the value function for a two-dimensional singular stochastic control problem.} SIAM J. Control Optim. 27 (1989), no. 4, 876--907. 

\bibitem{Trudinger} Trudinger, N. \emph{Fully nonlinear, uniformly elliptic equations under natural structure conditions.} Trans. Amer. Math. Soc. 278 (1983), no. 2, 751--769. 

\bibitem{Wiegner}  Wiegner, M. \emph{The $C^{1,1}$-character of solutions of second order elliptic equations with gradient constraint}. Comm. Partial Differential Equations 6 (1981), no. 3, 361--371.



\end{thebibliography}
\end{document}